\documentclass[11pt,letterpaper]{article}
\usepackage{amsmath,amssymb,amsthm}
\usepackage{MnSymbol} 
\usepackage{graphicx,color}
\usepackage[hyphens]{url}
\usepackage{dsfont}
\usepackage{booktabs}
\usepackage[square,sort,numbers]{natbib}
\usepackage[sort,nocompress,noadjust]{cite}
\usepackage[normalem]{ulem}
\usepackage{mathtools}
\usepackage[nameinlink,capitalize]{cleveref}
\usepackage{multirow}
\usepackage{algorithm}
\usepackage[noend]{algpseudocode}
\usepackage{xspace}
\usepackage{tikz}

\numberwithin{equation}{section}
\newtheorem{theorem}{Theorem}[section]
\newtheorem{cor}[theorem]{Corollary}
\newtheorem{lemma}[theorem]{Lemma}
\newtheorem{remark}[theorem]{Remark}
\newtheorem{prop}[theorem]{Proposition}
\newtheorem{obs}[theorem]{Observation}

\newtheorem{defin}[theorem]{Definition}

\newcommand{\cF}{\mathcal{F}}

\newcommand{\R}{\mathbb{R}}
\newcommand{\C}{\mathbb{C}}
\newcommand{\N}{\mathbb{N}}

\newcommand{\Rn}{\R^n}

\newcommand{\Rnn}{\R^{n \times n}}
\newcommand{\Rp}{\R_{\geq 0}}

\newcommand{\Rpn}{\R_{\geq 0}^n}
\newcommand{\Rpnn}{\R_{\geq 0}^{n \times n}}
\newcommand{\Rppn}{\R_{> 0}^n}
\newcommand{\Rppnn}{\R_{> 0}^{n \times n}}

\newcommand{\Prob}{\mathbb{P}}	
\newcommand*{\E}{\mathbb{E}}
\newcommand{\supp}{\operatorname{supp}}
\newcommand{\suppK}{\operatorname{supp}(K)}
\newcommand{\Hell}{\mathsf{H}}
\newcommand{\TV}{\mathsf{TV}}

\providecommand{\diag}{\operatorname{\mathbb{D}}}

\newcommand{\bone}{\mathbf{1}}
\newcommand{\zero}{\mathbf{0}}

\newcommand*{\Otilde}{\tilde{O}}

\DeclareMathOperator{\logdel}{\log\tfrac{1}{\delta}}
\DeclareMathOperator{\logtdel}{\log\tfrac{2}{\delta}}

\newcommand{\diamG}{d}
\newcommand*{\ep}{\varepsilon}
\newcommand*{\eps}{\varepsilon}

\newcommand{\plusminus}{\raisebox{.2ex}{$\scriptstyle\pm$}}
\DeclareMathOperator*{\argmin}{argmin}
\DeclareMathOperator*{\argmax}{argmax}
\renewcommand{\leq}{\leqslant}
\renewcommand{\geq}{\geqslant}
\DeclareMathOperator{\half}{\frac{1}{2}}

\providecommand{\abs}[1]{\left\lvert#1\right\rvert}

\newcommand*{\BALK}{\textsc{BAL}(K)}
\newcommand*{\ABAL}{\textsc{ABAL}}
\newcommand*{\ABALKeps}{\textsc{ABAL}(K,\eps)}

\providecommand{\varnorm}{\operatorname{var}}
\providecommand{\var}[1]{\lVert{#1}\rVert_{\varnorm}}

\newcommand*{\Kmax}{K_{\max}}
\newcommand*{\Kmin}{K_{\min}}
\newcommand*{\condK}{\kappa}

\newcommand*{\xt}{x^{(t)}}
\newcommand*{\xtm}{x^{(t-1)}}
\newcommand*{\xtp}{x^{(t+1)}}
\newcommand*{\xzero}{x^{(0)}}
\newcommand*{\xone}{x^{(1)}}
\newcommand*{\xtwo}{x^{(2)}}

\usepackage[margin=1in]{geometry}

\makeatletter
\def\blfootnote{\gdef\@thefnmark{}\@footnotetext}
\makeatother

\begin{document}
	\title{Near-linear convergence of the Random Osborne algorithm for Matrix Balancing}
	\author{Jason M. Altschuler \and Pablo A. Parrilo}
	\date{}
	\maketitle

\blfootnote{The authors are with the Laboratory for Information and Decision Systems (LIDS), Massachusetts Institute of Technology, Cambridge MA 02139. Work partially supported by NSF AF 1565235, NSF Graduate Research Fellowship 1122374, and a TwoSigma PhD Fellowship.}

\begin{abstract}
	We revisit Matrix Balancing, a pre-conditioning task used ubiquitously for computing eigenvalues and matrix exponentials. Since 1960, Osborne's algorithm has been the practitioners' algorithm of choice and is now implemented in most numerical software packages. However, its theoretical properties are not well understood. Here, we show that a simple random variant of Osborne's algorithm converges in near-linear time in the input sparsity. Specifically, it balances $K \in \Rpnn$ after $O(m \eps^{-2} \log \condK)$ arithmetic operations in expectation and with high probability, where $m$ is the number of nonzeros in $K$, $\eps$ is the $\ell_1$ accuracy, and $\condK = \sum_{ij} K_{ij} / ( \min_{ij : K_{ij} \neq 0} K_{ij})$ measures the conditioning of $K$. Previous work had established near-linear runtimes either only for $\ell_2$ accuracy (a weaker criterion which is less relevant for applications), or through an entirely different algorithm based on (currently) impractical Laplacian solvers.
	\par We further show that if the graph with adjacency matrix $K$ is moderately well-connected---e.g., if $K$ has at least one positive row/column pair---then Osborne's algorithm initially converges exponentially fast, yielding an improved runtime $O(m \eps^{-1} \log \condK)$. We also address numerical precision by showing that these runtime bounds still hold when using $O(\log(n\condK/\eps))$-bit numbers.
	\par Our results are established through an intuitive potential argument that leverages a convex optimization perspective of Osborne's algorithm, and relates the per-iteration progress to the current imbalance as measured in Hellinger distance. Unlike previous analyses, we critically exploit log-convexity of the potential. Notably, our analysis extends to other variants of Osborne's algorithm: along the way, we also establish significantly improved runtime bounds for cyclic, greedy, and parallelized variants of Osborne's algorithm.
\end{abstract}

\section{Introduction}\label{sec:intro}
Let $\bone$ denote the all-ones vector in $\Rn$. A nonnegative square matrix $A \in \Rpnn$ is said to be \emph{balanced} if its row sums $r(A) := A\bone$ equal its column sums $c(A) := A^T\bone$, i.e.
\begin{align}
r(A) = c(A).
\label{eq-def:balance}
\end{align}
This paper revisits the classical problem of \emph{Matrix Balancing}---sometimes also called \textit{diagonal similarity scaling} or \textit{line-sum-symmetric scaling}---which asks: given a nonnegative matrix $K \in \Rpnn$, find a positive diagonal matrix $D$ (if one exists\footnote{$K$ can be balanced if and only if $K$ is irreducible~\citep{EavHofRotSch85}. This can be efficiently checked in linear time~\citep{TarjanStronglyConnected}.
}) such that $A := DKD^{-1}$ is balanced.
\par Matrix Balancing is a fundamental problem in numerical linear algebra, scientific computing, and theoretical computer science with many applications and an extensive literature dating back to 1960. A particularly celebrated application of Matrix Balancing is pre-conditioning matrices before linear algebraic computations such as eigenvalue decomposition~\citep{Osborne60,ParRei69} and matrix exponentiation~\citep{Ward77,Higham05}. The point is that performing these linear algebra tasks on a balanced matrix can drastically improve numerical stability and readily recovers the desired answer on the original matrix~\citep{Osborne60}. Moreover, in practice, the runtime of (approximate) Matrix Balancing is essentially negligible compared to the runtime of these downstream tasks~\citep[\S11.6.1]{PreTeuVetFla07}. 
The ubiquity of these applications has led to the implementation of Matrix Balancing in most linear algebra software packages, including EISPACK~\citep{Eispack}, LAPACK~\citep{Lapack}, R~\citep{Rbal}, and MATLAB~\citep{MATLABbal}. 
In fact, Matrix Balancing is performed by default in the command for eigenvalue decomposition in MATLAB~\citep{MATLABeig} and in the command for matrix exponentation for R~\citep{Rexpm}. Matrix Balancing also has other diverse applications in economics~\citep{SchZen90}, information retrieval~\citep{Tom03}, and combinatorial optimization~\citep{AltPar20mmc}.

\par In practice, Matrix Balancing is performed approximately rather than exactly, since this can be done efficiently and typically suffices for applications. Specifically, in the \emph{approximate Matrix Balancing problem}, the goal is to compute a scaling $A := DKD^{-1}$ that is $\eps$\emph{-balanced} in the $\ell_1$ sense, i.e.,
\begin{align}
\frac{\|r(A) - c(A)\|_1}{\sum_{ij} A_{ij}} \leq \eps.
\label{eq-def:balance-approx}
\end{align}

\begin{remark}[$\ell_1$ versus $\ell_2$ imbalance]
	Several papers~\citep{KalKhaSho97,OstRabYou17} study approximate Matrix Balancing with $\ell_2$ norm imbalance---rather than $\ell_1$ as done here in~\eqref{eq-def:balance-approx} and in  e.g.,~\citep{NemRot99}---for what appears to be essentially historical reasons.
	Here, we focus solely on the $\ell_1$ imbalance as it appears to be more useful for applications---e.g., it is critical for near-linear time approximation of the Min-Mean-Cycle problem~\citep{AltPar20mmc}---in large part due to its natural interpretations in both probabilistic problems (as total variation imbalance) and graph theoretic problems (as netflow imbalance)~\citep[Remarks 2.1 and 5.8]{AltPar20mmc}.\footnote{The analogous observation has also been made for the intimately related problem of Matrix Scaling. For example, the $\ell_1$ norm is pivotal there for applications including Optimal Transport~\citep{AltWeeRig17} and Bipartite Perfect Matching~\citep{ChaKha18}.}
	Note also that the approximate balancing criterion~\eqref{eq-def:balance-approx} is significantly easier to achieve\footnote{
		As a simple concrete example, let $n$ be even and consider 
		the $n \times n$ matrix $A$ which is $0$ everywhere except is the identity on the top right $n/2 \times n/2$ block. 
		Note that $r(A)/\sum_{ij} A_{ij} = [\tfrac{2}{n}\bone_{n/2}, \zero_{n/2}]^T$ and $c(A)/\sum_{ij} A_{ij} = [\zero_{n/2}, \tfrac{2}{n}\bone_{n/2}]^T$. 
		Thus $A$ is \emph{as unbalanced as possible} in $\ell_1$ norm since $\|r(A) - c(A)\|_1/\sum_{ij}A_{ij} = 2$; however, $A$ is very well balanced in $\ell_2$ norm since $\|r(A) - c(A)\|_2 / \sum_{ij}A_{ij} = 2/\sqrt{n}$ is vanishingly small.
	} for $\ell_2$ imbalance than $\ell_1$: in fact, any matrix can be balanced to constant $\ell_2$ error by only rescaling a vanishing $1/n$ fraction of the entries~\citep{OstRabYou17}, whereas this is impossible for the $\ell_1$ norm. 
(Note that this issue of which norm to measure imbalance should \emph{not} be confused with the $\ell_p$ Matrix Balancing problem, see Remark~\ref{rem:lp-bal}.)
\end{remark}

\subsection{Previous algorithms}\label{ssec:intro:prev}

The many applications of Matrix Balancing have motivated an extensive literature focused on solving it efficiently. 
However, there is still a large gap between theory and practice, and several key issues remain. We overview the relevant previous results below.

\subsubsection{Practical state-of-the-art}

Ever since its invention in 1960, \emph{Osborne's algorithm} has been the algorithm of choice for practitioners~\citep{Osborne60,ParRei69}. Osborne's algorithm is a simple iterative algorithm which initializes $D$ to the identity (i.e., no balancing), and then in each iteration performs an \emph{Osborne update} on some update coordinate $k \in [n]$, in which $D_{kk}$ is updated to $\sqrt{c_k(A)/r_k(A)} D_{kk}$ so that the $k$-th row sum $r_k(A)$ and $k$-th column sum $c_k(A)$ of the current balancing $A = DKD^{-1}$ agree.\footnote{We assume throughout that the diagonal of $K$ is zero. This ensures that the Osborne update makes the row and column sums agree. This assumption is without loss of generality because if $D$ $\ep$-balances $K$ with zeroed-out diagonal, then it also $\eps$-balances $K$.} A more precise statement is in Algorithm~\ref{alg:osb} later.

\par The classical version of Osborne's algorithm, henceforth called \emph{Round-Robin Cyclic Osborne}, chooses the update coordinates by repeatedly cycling through $\{1,\dots, n\}$.
This algorithm\footnote{To be precise, following~\citep{ParRei69}, some implementations have two minor modifications: a pre-processing step where $K$ is permuted to a triangular block matrix with irreducible diagonal blocks; and a restriction of the entries of $D$ to exact powers of the radix base. We presently ignore these minor modifications since the former is easily performed in linear-time~\citep{TarjanStronglyConnected}, and the latter is solely to safeguard against numerical precision issues in practice.} performs remarkably well in practice and is the implementation of choice in most linear algebra software packages.

\par Despite this widespread adoption of Osborne's algorithm, a theoretical understanding of its convergence has proven to be quite challenging: indeed, non-asymptotic convergence bounds (i.e., runtime bounds) were not known for nearly $60$ years until the breakthrough $2017$ paper~\citep{OstRabYou17}. The paper~\citep{OstRabYou17} shows\footnote{Note that in~\citep{OstRabYou17}, bounds are written for the $\ell_2$ balancing criteria; see the discussion after~\eqref{eq-def:balance-approx}.\label{ft:ost}} that Round-Robin Cyclic Osborne computes an $\eps$-balancing after $O(m n^2 \eps^{-2} \log \condK)$ arithmetic operations, where $m$ is the number of nonzeros in $K$, and $\condK := (\sum_{ij} K_{ij})/(\min_{ij : K_{ij} \neq 0} K_{ij})$. They also show faster $\tilde{O}(n^2 \eps^{-2} \log \condK)$  runtimes for two variants of Osborne's algorithm which choose update coordinates in different orders than cyclically. Here and henceforth, the $\Otilde$ notation suppresses polylogarithmic factors in $n$ and $\eps^{-1}$. The first variant, which we call \emph{Greedy Osborne}, chooses the coordinate with maximal imbalance as measured by $\argmax_k (\sqrt{r_k(A)} - \sqrt{c_k(A)})^2$. They show that Greedy Osborne's runtime dependence on $\eps$ can be improved from $\eps^{-2}$ to $\eps^{-1}$; however, this comes at the high cost of an extra factor of $n$. A disadvantage of Greedy Osborne is that it has numerical precision issues and requires operating on $O(n \log \condK)$-bit numbers.
The second variant, which we call \emph{Weighted Random Osborne}, chooses coordinate $k$ with probability proportional to $r_k(A) + c_k(A)$, and can be implemented using $O(\log(n\kappa/\eps))$-bit numbers.

\par Collectively, these runtime bounds are fundamental results since they establish that Osborne's algorithm has polynomial runtime in $n$ and $\eps^{-1}$, and moreover that variants of it converge in roughly $\Otilde(n^2\eps^{-2})$ time for matrices satisfying $\log \condK = \Otilde(1)$---henceforth called \emph{well-conditioned matrices}. However, these theoretical runtime bounds are still much slower than both Osborne's rapid empirical convergence and the state-of-the-art theoretical algorithms described below. 
\par Two remaining open questions that this paper seeks to address are: 
\begin{enumerate}
	\item \textbf{Near-linear runtime\footnote{Throughout, we say a runtime is near-linear if it is $O(m)$, up to polylogarithmic factors in $n$ and polynomial factors in the inverse accuracy $\eps^{-1}$.}.}
	 Does (any variant of) Osborne's algorithm have near-linear runtime in the input sparsity $m$? The fastest known runtimes scale as $n^2$, which is significantly slower for sparse problems.
	\item \textbf{Scalability in accuracy.} The fastest runtimes for (any variant of) Osborne's algorithm scale poorly in the accuracy as $\eps^{-2}$. (Except Greedy Osborne, for which it is only known that $\eps^{-2}$ can be replaced by $\eps^{-1}$ at the high cost of an extra factor of $n$.) Can this be improved?
\end{enumerate}

\subsubsection{Theoretical state-of-the-art}
A separate line of work leverages sophisticated optimization techniques to solve a convex optimization problem equivalent to Matrix Balancing. These algorithms have $\log \eps^{-1}$ dependence on the accuracy, but are not practical (at least currently) due to costly overheads required by their significantly more complicated iterations. This direction originated in~\citep{KalKhaSho97}, which showed that the Ellipsoid algorithm produces an approximate balancing in $\tilde{O}(n^4 \log( (\log \condK) / \eps))$ arithmetic operations on $O(\log(n\condK/\eps))$-bit numbers.
Recently,~\citep{CohMadTsiVla17}\footnote{Similar runtimes were also developed by~\citep{ZhuLiOliWig17}.} gave an Interior Point algorithm with runtime $\Otilde(m^{3/2}\log(\condK/\eps))$ and a Newton-type algorithm with runtime 
$\tilde{O}(m \diamG \log^2 (\kappa /\eps) \log \kappa)$,
where $d$ denotes the diameter of the directed graph $G_K$ with vertices $[n]$ and edges $\{(i,j) : K_{ij} > 0\}$~\citep[Theorem 4.18, Theorem 6.1, and Lemma 4.24]{CohMadTsiVla17}.
Note that under the condition that $K$ is a \emph{well-connected matrix}---by which we mean that $G_K$ has polylogarithmic diameter $d = \Otilde(1)$---then this latter algorithm has near-linear runtime in the input sparsity $m$. However, these algorithms heavily rely upon near-linear time Laplacian solvers, for which practical implementations are not known.

\subsection{Contributions}\label{ssec:intro:contributions}

\paragraph*{Random Osborne converges in near-linear time.}
Our main result (Theorem~\ref{thm:osb-rand}) addresses the two open questions above by showing that a simple random variant of the ubiquitously used Osborne's algorithm has runtime that is (i) near-linear in the input sparsity $m$, and also (ii) linear in the inverse accuracy $\eps^{-1}$ for well-connected inputs. Property (i) amends the aforementioned gap between theory and practice that the fastest known runtime of Osborne's algorithm scales as $n^2$~\citep{OstRabYou17}, while a different, impractical algorithm has theoretical runtime which is (conditionally) near-linear in $m$~\citep{CohMadTsiVla17}. Property (ii) shows that improving the runtime dependence in $\eps$ from $\eps^{-2}$ to $\eps^{-1}$ does not require paying a costly factor of $n$ (c.f.,~\citep{OstRabYou17}).

\par Specifically, we propose a variant of Osborne's algorithm---henceforth called \emph{Random Osborne}\footnote{
	Not to be confused with the different randomized variant of Osborne's algorithm in~\citep[\S5]{OstRabYou17}, which draws coordinates with non-uniform probabilities. We call that algorithm Weighted Random Osborne to avoid confusion.
}--- which chooses update coordinates uniformly at random, and show the following.

\begin{theorem}[Informal version of Theorem~\ref{thm:osb-rand}]\label{thm:intro:rand}
	Random Osborne solves the approximate Matrix Balancing problem on input $K \in \Rpnn$ to accuracy $\eps > 0$ after
	\begin{align}
	O\left(\frac{m}{\eps} \left(\frac{1}{\eps} \wedge \diamG\right) \log \condK \right),
	\label{eq-thm:intro:rand}
	\end{align}
	arithmetic operations, both in expectation and with high probability.
\end{theorem}

\begin{table}
	\centering
	\begin{tabular}{|c|c|c|}
		\hline
		\textbf{Variant} & \textbf{Best runtime bound (arithmetic operations)} & \textbf{Polylog bits} \\ [0.5ex]  \hline \hline
		Cyclic (Round-Robin)
		&
		$\Otilde(mn^2/\eps^2)$~\citep{OstRabYou17} 
		&
		No
		\\ \hline
		Cyclic (Random-Reshuffle)
		&
		$\boldsymbol{\Otilde(m n / \eps)}$ \textbf{(Theorem~\ref{thm:osb-cyc})}
		&
		\textbf{Yes (Theorem~\ref{thm:osb-all:bits})}
		\\ \hline
		Weighted Random
		&
		$\Otilde(n^2/\eps^{2})$~\citep{OstRabYou17}
		& 
		Yes~\citep{OstRabYou17}
		\\ \hline
		Greedy
		&
		$\Otilde((n^2/\eps^2) \wedge (n^3/\eps))$~\citep{OstRabYou17} $\longrightarrow$
		$\boldsymbol{\Otilde(n^2 / \eps)}$ \textbf{(Theorem~\ref{thm:osb-greedy})}
		&
		No
		\\ \hline
		Random
		& 
		$\boldsymbol{\Otilde(m /\eps)}$ \textbf{(Theorem~\ref{thm:osb-rand})}
		&
		\textbf{Yes (Theorem~\ref{thm:osb-all:bits})}
		\\ \hline
	\end{tabular}
	\caption{
		Variants of Osborne's algorithm for balancing a matrix $K \in \Rpnn$ with $m$ nonzeros to $\eps$ $\ell_1$ accuracy. For simplicity, here $K$ is assumed well-conditioned (i.e., $\log \condK = \Otilde(1)$) and well-connected (i.e., $\diamG = \Otilde(1))$; see the main text for detailed dependence on $\log \condK$ and $\diamG$. 
		Note that in~\citep{OstRabYou17}, bounds are written for the $\ell_2$ criterion; see the discussion after~\eqref{eq-def:balance-approx}.
		See the main text for descriptions of each variant, and also \S\ref{ssec:prelim:cd} for more details on Random-Reshuffle Cyclic, Greedy, and Random Osborne.
		Our new bounds are in bold.
	}
	\label{tab:bal}
\end{table}

\par We make several remarks about Theorem~\ref{thm:intro:rand}. First, we interpret the runtime~\eqref{eq-thm:intro:rand}. This is the minimum of $O(m\eps^{-2} \log \condK)$ and $O(m \diamG \eps^{-1} \log \condK)$. The former is near-linear in $m$. The latter is too if $G_K$ has polylogarithmic diameter $\diamG = \Otilde(1)$---important special cases include matrices $K$ containing at least one strictly positive row/column pair (there, $\diamG = 1$), and matrices with random sparsity patterns (there, $\diamG = \Otilde(1)$ with high probability, see, e.g.,~\citep[Theorem 10.10]{Bol01}). Note that the complexity of Matrix Balancing is intimately related to the connectivity of $G_K$: indeed, $K$ can be balanced if and only if $G_K$ is strongly connected (i.e., if and only if $\diamG$ is finite)~\citep{Osborne60}. Intuitively, the runtime dependence on $\diamG$ is a quantitative measure of ``how balanceable'' the input $K$ is. 
\par We note that the high probability bound in Theorem~\ref{thm:intro:rand} has tails that decay exponentially fast. This is optimal with our analysis, see Remark~\ref{rem:prob-subexp}.

\par Next, we comment on the $\log \condK$ term in the runtime. This term appears in all other state-of-the-art runtimes~\citep{CohMadTsiVla17,OstRabYou17} and is mild: indeed, $\log \condK \leq \log m + \log (\max_{ij} K_{ij}/ \min_{ij : K_{ij} > 0} K_{ij})$, where the former summand is $\tilde{O}(1)$---hence why the runtime is \textit{near}-linear---and the latter is the input size for the entries of $K$. In particular, if $K$ has quasi-polynomially bounded entries, then $\log \condK = \Otilde(1)$. 

\par Next, we compare to existing runtimes. Theorem~\ref{thm:intro:rand} gives a faster runtime than any existing practical algorithm, see Table~\ref{tab:bal}. If comparing to the (impractical) algorithm of~\citep{CohMadTsiVla17} on a purely theoretical plane, neither runtime dominates the other, and which is faster depends on the precise parameter regime:~\citep{CohMadTsiVla17} is better for high accuracy solutions\footnote{We remark that in practical applications of Matrix Balancing such as pre-conditioning, low accuracy solutions typically suffice. Indeed, this is a motivation of the commonly used variant of Osborne's algorithm which restricts entries of the scaling $D$ to exact powers of the radix base~\citep{ParRei69}.}, while Random Osborne has better dependence on the conditioning $\condK$ of $K$ and the connectivity $\diamG$ of $G_K$. 

Finally, we remark about bit-complexity. In \S\ref{sec:bits}, we show that with only minor modification, Random Osborne is implementable using numbers with only logarithmically few $O(\log (n \condK / \eps))$ bits; see Theorem~\ref{thm:osb-all:bits} for formal statement.

\paragraph*{Simple, streamlined analysis for different Osborne variants.} We prove Theorem~\ref{thm:intro:rand} using an intuitive potential argument (overviewed in \S\ref{ssec:intro:overview} below). An attractive feature of this argument is that with only minor modification, it adapts to other Osborne variants. We elaborate below; see also Tables~\ref{tab:bal} and~\ref{tab:par} for summaries of our improved rates.

\par \underline{Greedy Osborne.} We show an improved runtime for Greedy Osborne where the $\eps^{-2}$ dependence is improved to $\eps^{-1}$ at the cost of $\diamG$ (rather than a full factor of $n$ as in~\citep{OstRabYou17}). Specifically, in Theorem~\ref{thm:osb-greedy}, we show convergence after $O(n^2\eps^{-1} (\eps^{-1} \wedge \diamG) \log \condK)$ arithmetic operations, which improves upon the previous best $O(n^2\eps^{-1} \log n \cdot (\eps^{-1} \log \condK \wedge n \log (\kappa/\eps)))$ from~\citep{OstRabYou17}. (The other improved $\log n$ factor comes from simplifying the data structure used for efficient greedy updates, see Remark~\ref{rem:greedy-amortized}.)

\par \underline{Random-Reshuffle Cyclic Osborne.} 
We analyze Random-Reshuffle Cyclic Osborne, which is the variant of Osborne's algorithm that cycles through all $n$ indices using a fresh random permutation in each cycle. 
We show that this algorithm converges after $O(m n \eps^{-1} (\eps^{-1} \wedge \diamG) \log \condK)$ arithmetic operations (Theorem~\ref{thm:osb-cyc}). Previously, the only known runtime bound for any variant of Osborne with ``cyclic'' updates in the sense that each index is updated exactly once per epoch, was the $O(mn^2\eps^{-2} \log \condK)$ runtime bound for Round-Robin Cyclic Osborne~\citep{OstRabYou17}. Although the version of Cyclic Osborne we study is different than the one studied in~\citep{OstRabYou17}, we note that our runtime bound is a factor of $n$ faster, and additionally a factor of $1/\eps$ faster if the matrix is well-connected.
Moreover, 
we show that Random-Reshuffle Cyclic Osborne can be implemented on numbers with $O(\log(n\kappa/\eps))$-bit numbers (Theorem~\ref{thm:osb-all:bits}), whereas the analysis of Round-Robin Cyclic Osborne in~\citep{OstRabYou17} requires $O(n \log \kappa)$-bit numbers.

\par \underline{Parallelized Osborne.} 
We also show fast convergence for the analogous greedy, cyclic, and random variants of a parallelized version of Osborne's algorithm that is recalled in \S\ref{ssec:prelim:par}. These runtimes bounds are summarized in Table~\ref{tab:par}. 
Our main result here is that---modulo at most a single $\log n$ factor arising from the conditioning $\log \condK$ of the input---Random Block Osborne converges after (i) only a linear number $O(\tfrac{p}{\eps}(\tfrac{1}{\eps} \wedge \diamG) \log \condK)$ of synchronization rounds in the size $p$ of the dataset partition; and (ii) the same amount of total work as its non-parallelized counterpart Random Osborne, which is in particular near-linear in $m$ (see Theorem~\ref{thm:intro:rand} and the ensuing discussion). Property (i) shows that, when giving an optimal coloring of $G_K$, Random Osborne converges in linear time in the chromatic number $\chi(G_K)$ of $G_K$ (see \S\ref{ssec:prelim:par} for further details). Property (ii) shows that the speedup of parallelization comes at no cost in the total work.

\begin{table}
	\centering
	\begin{tabular}{|c|c|c|c|}
		\hline
		\textbf{Variant} & \textbf{Best runtime bound (rounds)} & \textbf{Total work} & \textbf{Polylog bits} \\ [0.5ex]  \hline \hline
		Cyclic Block (Random-Reshuffle)
		&
		$\Otilde(p^2/\eps)$
		&
		$\Otilde(m p/\eps)$
		&
		Yes 
		\\ \hline
		Greedy Block
		&
		$\Otilde(p/\eps)$
		&
		$\Otilde(mp/\eps)$
		&
		No
		
		\\ \hline
		Random Block
		&
		$\Otilde(p/\eps)$
		&
		$\Otilde(m/\eps)$
		&
		Yes 
		\\ \hline
	\end{tabular}
	\caption{
		Parallelized variants of Osborne's algorithm for balancing a matrix $K \in \Rpnn$ with $m$ nonzeros to $\eps$ $\ell_1$ accuracy, given a partitioning of the dataset into $p$ blocks (see \S\ref{ssec:prelim:par} for details). For simplicity, here $K$ is assumed well-conditioned (i.e., $\log \condK = \Otilde(1)$) and well-connected (i.e., $\diamG = \Otilde(1))$; see the main text for detailed dependence on $\log \condK$ and $\diamG$. 
		All results are ours. The runtime and work bounds are in Theorem~\ref{thm:par:all}, and the bit-complexity bounds are in Theorem~\ref{thm:osb-all:bits}. 
	}
	\label{tab:par}
\end{table}

\subsection{Overview of approach}\label{ssec:intro:overview}

We establish all of our runtime bounds with essentially the same potential argument. Below, we first sketch this argument for Greedy Osborne, since it is the simplest. Next, we describe the modifications for Random Osborne---the argument is identical modulo probabilistic tools which, albeit necessary for a rigorous analysis, are not the heart of the argument. We then outline the analysis for Random-Reshuffle Cyclic Osborne, which follows as a straightforward corollary.
We then briefly remark upon the very minor modifications required for the parallelized Osborne variants.

\par For all variants, the potential we use is $D \mapsto \Phi(D) - \inf_{D^*} \Phi(D^*)$, where for a positive diagonal matrix $D$, we write $\Phi(D) = \log \sum_{ij} A_{ij}$ to denote the logarithm of the sum of the entries of the current balancing $A = DKD^{-1}$. Minimizing this potential function is well-known to be equivalent to Matrix Balancing; details in the Preliminaries section \S\ref{ssec:prelim:convex}. Note also that Osborne's algorithm is equivalent to Exact Coordinate Descent on this function---which, importantly, is convex after a re-parameterization; see \S\ref{ssec:prelim:cd}. In the interest of accessibility, the below overview describes our approach at an informal level that does not require further background. Later, \S\ref{sec:prelim} provides these preliminaries, and \S\ref{sec:pot} gives the technical details of the potential argument.

\subsubsection{Argument for Greedy Osborne}\label{sssec:intro:overview-greedy}
Here we sketch the $O(n^2 \eps^{-1} (\eps^{-1} \wedge \diamG) \log \condK)$ runtime we establish for Greedy Osborne in \S\ref{sec:greedy}. Since each Greedy Osborne iteration takes $O(n)$ arithmetic operations (see \S\ref{ssec:prelim:cd}), it suffices to bound the number of iterations by $O(n \eps^{-1} (\eps^{-1} \wedge \diamG) \log \condK)$.
\par The first step is relating the per-iteration progress of Osborne's algorithm to the imbalance of the current balancing---as measured in \emph{Hellinger distance} $\Hell(\cdot,\cdot)$. Specifically, we show that an Osborne update decreases the potential function by at least
\begin{align}
(\text{per-iteration decrease in potential})
\gtrsim
\frac{\Hell^2 \left( r(P), c(P) \right)}{n},
\label{eq-ove:hell}
\end{align}
where $P = A/\sum_{ij} A_{ij}$ is the normalization of the current scaling $A = DKD^{-1}$. Note that since $P$ is normalized, its marginals $r(P)$ and $c(P)$ are both probability distributions.
\par The second step is lower bounding this Hellinger imbalance $\Hell^2 \left( r(P), c(P) \right)$ by something large, so that we can argue that each iteration makes significant progress. Following is a simple such lower bound that yields an $O(n^2 \eps^{-2} \log \condK)$ runtime bound. Modulo small constant factors: a standard inequality in statistics lower bounds Hellinger distance by $\ell_1$ distance (a.k.a. total variation distance), and the $\ell_1$ distance is by definition at least $\eps$ if the current iterate is not $\eps$-balanced (see~\eqref{eq-def:balance-approx}). Therefore 
\begin{align}
(\text{per-iteration decrease in potential})
\gtrsim
\frac{\eps^2}{n}
\label{eq-ove:1}
\end{align}
for each iteration before convergence. Since the potential is initially not very large (at most $\log \condK$, see Lemma~\ref{lem:pot-init}) and by construction always nonnegative, the total number of iterations before convergence is therefore at most $n \eps^{-2} \log \condK$.

\par The key to the improved bound is an extra inequality that shows that \emph{the per-iteration decrease is very large when the potential is large}. Specifically, this inequality---which has a simple proof using convexity of the potential---implies the following improvement of~\eqref{eq-ove:1}
\begin{align}
(\text{per-iteration decrease in potential})
\gtrsim
\frac{1}{n} \left[\frac{\text{(current potential)}}{R} \vee \eps \right]^2
\label{eq-ove:2}
\end{align}
where $R = \diamG \log \condK$. The per-iteration decrease is thus governed by the maximum of these two quantities. In words, the former ensures a \emph{relative improvement} in the potential, and the latter ensures an \emph{additive improvement}. Which is bigger depends on the current potential: the former dominates when the potential is $\Omega(\eps R)$, and the latter for $O(\eps R)$. 
It can be shown that both ``phases'' require $O(n \eps^{-1} \diamG \log \condK)$ iterations, yielding the desired improved rate (details in \S\ref{sec:greedy}).

\subsubsection{Argument for Random Osborne}\label{sssec:intro:overview-rand}

The argument for Random Osborne is nearly identical, except for two minor changes. The first change is the per-iteration potential decrease. All the same bounds hold (i.e.,~\eqref{eq-ove:hell},~\eqref{eq-ove:1}, and~\eqref{eq-ove:2}), except that they are now \emph{in expectation} rather than deterministic. Nevertheless, this large expected progress is sufficient to obtain the same iteration-complexity bound. Specifically, an expected bound on the number of iterations is proved using Doob's Optional Stopping Theorem, and a h.p. bound using a martingale Chernoff bound (details in \S\ref{sssec:conv:runtime}).
\par The second change is the per-iteration runtime: it is faster in expectation.

\begin{obs}[Per-iteration runtime of Random Osborne]\label{obs:osb-rand:iter}
	An iteration of Random Osborne requires $O(m/n)$ arithmetic operations in expectation.
\end{obs}
\begin{proof}
	The number of arithmetic operations required by an Osborne update on coordinate $k$ is proportional to the number of nonzero entries on the $k$-th row and column of $K$. Since Random Osborne draws $k$ uniformly from $[n]$, this number of nonzeros is $2m/n$ in expectation.
\end{proof}

Note that this per-iteration runtime is $n^2/m$ times faster than Greedy Osborne's. 
This is why our bound on the total runtime of Random Osborne is roughly $O(m)$, whereas for Greedy Osborne it is $O(n^2)$. 
\par A technical nuance is that arguing a final runtime bound from a per-iteration runtime and an iteration-complexity bound is a bit more involved for Random Osborne. This is essentially because the number of iterations is not statistically independent from the per-iteration runtimes. For Greedy Osborne, the final runtime is bounded simply by the product of the per-iteration runtime and the number of iterations. We show a similar bound for Random Osborne in expectation via a slight variant of Wald's inequality, and w.h.p. via a Chernoff bound; details in \S\ref{sssec:conv:decrease}.

\subsubsection{Argument for Random-Reshuffle Cyclic Osborne}\label{sssec:intro:overview-cyc}
Analyzing Cyclic Osborne (either Round-Robin or Random-Reshuffle) is difficult because the improvement of an Osborne update is significantly affected by the previous Osborne updates in the cycle---and this effect is difficult to track. We observe that our improved analysis for Random Osborne implies, as a straightforward corollary, a fast runtime for Random-Reshuffle Cyclic Osborne. Specifically, since Osborne updates monotonically improve the potential, the per-cycle improvement of Random-Reshuffle Cyclic Osborne is at least the improvement of the first iteration of the cycle, which equals the improvement of a single iteration of Random Osborne. This implies that Random-Reshuffle Cyclic Osborne requires at most $n$ times more iterations than Random Osborne. Details in \S\ref{sec:cyclic}. We remark that while arguing about a cycle only through its first iteration is clearly quite pessimistic, improvements seem difficult. A similar difficulty occurs for the analysis of Cyclic Coordinate Descent in more general convex optimization setups.

\subsubsection{Argument for parallelized Osborne}\label{sssec:intro:overview-block}

The argument for the parallelized variants of Osborne are nearly identical to the arguments for their non-parallelized counterparts, described above. Specifically, the main difference for the random and greedy variants is just that in the bounds~\eqref{eq-ove:hell},~\eqref{eq-ove:1}, and~\eqref{eq-ove:2}, the $1/n$ factor is improved to $1$ over the partitioning size $p$. The same argument then results in a final runtime that is sped up by this factor of $n/p$. The only difference for analyzing the Random-Reshuffle Cyclic variant is that here, the analogous coupling argument only gives a slowdown of $p$ rather than $n$. Details in \S\ref{sec:par}.

\subsubsection{Key differences from previous approaches}
	The only other polynomial-time analysis of Osborne's algorithm also uses a potential argument~\citep{OstRabYou17}. However, our argument differs in several key ways---which enables much tighter bounds as well as a simpler argument that extends to many variants of Osborne's algorithm. Notably, their proof of Lemma 3.1 (which is where they show that each iteration of Greedy Osborne makes progress; c.f. our Lemma~\ref{lem:pot-dec:greedy}) is specifically tailored to Greedy Osborne\footnote{
		Specifically, to prove their Lemma 3.1,~\citep{OstRabYou17} uses in (3.6) the inequality $\max_{i \in [n]} a_i / b_i \geq ( \tfrac{1}{n} \sum_{j=1}^n a_j) / (\tfrac{1}{n} \sum_{j=1}^n b_j)$ for positive $a_1, \dots, a_n, b_1, \dots, b_n$. Extending their analysis of Greedy Osborne to Random Osborne would require replacing $\max_{i \in [n]} a_i / b_i$ by $\tfrac{1}{n}\sum_{i=1}^n a_i / b_i$ in that inequality; however, this inequality is false because an average of ratios is in general incomparable to the ratio of averages. We bypass this obstacle by arguing in such a way that the quantity we need to bound is not a fraction, since such an analysis readily extends to Random Osborne by linearity of expectation (see~\eqref{eq-lpdg:3} and Lemma~\ref{lem:pot-dec:rand}).} and seems unextendable to other variants such as Random Osborne.
	In particular, this precludes obtaining the near-linear runtime shown in this paper.
	Another key difference is that they do \emph{not} use convexity of their potential (explicitly written on~\citep[page 157]{OstRabYou17}), whereas we exploit not only convexity but also \emph{log-convexity} (note our potential is the logarithm of theirs).
	Specifically, they use~\citep[Lemma 2.2]{OstRabYou17} to improve $\eps^{-2}$ to $\eps^{-1}$ dependence at the cost of an extra factor of $n$, whereas here we show a significantly tighter bound (see the proof of Proposition~\ref{prop:hell-lb}) that saves this factor of $n$ for well-connected graphs by exploiting log-convexity of their potential.

\subsection{Other related work}
We briefly remark about several tangentially related lines of work. Reference~\citep{Chen00} gives heuristics for speeding up Osborne's algorithm on sparse matrices in practice, but does not provide runtime bounds. Reference~\citep{OstRabYou18} gives a more complicated version of Osborne's algorithm that obtains a stricter approximate balancing in a polynomial (albeit less practical) runtime of roughly $\Otilde(n^{19} \eps^{-4} \log^4 \condK)$. Reference~\citep{MaiBat19} gives an asynchronous distributed version of Osborne's algorithm with applications to epidemic suppression.

\begin{remark}[Fast Coordinate Descent]
	Since Osborne's algorithm is Exact Coordinate Descent on a certain associated convex optimization problem (details in \S\ref{ssec:prelim:cd}), it is natural to ask what runtimes the extensive literature on Coordinate Descent implies for Matrix Balancing. However, applying general-purpose bounds on Coordinate Descent out-of-the-box gives quite pessimistic runtime bounds for Matrix Balancing\footnote{E.g., consider applying the state-of-the-art guarantees of~\citep{NesSti17,AllQuRicYua16} for accelerated Coordinate Descent algorithms (which, note also, do \emph{not} correspond exactly to Osborne's algorithm since they do not perform exact coordinate minimization). These bounds apply to Random Coordinate Descent with judiciously chosen non-uniform sampling probabilities, and yield an iteration bound of $(\sum_{i=i}^n \sqrt{L_i})\delta^{-1/2}\|x^*\|_2$ for minimizing $\Phi$ (defined in \S\ref{ssec:prelim:convex}) to $\delta$ additive accuracy, where $L_i$ is the smoothness of $\Phi$ on coordinate $i$. By~\citep[Corollary 2]{KalKhaSho97} and Cauchy-Schwarz, $\delta = O(\eps^2/n)$ ensures that such a $\delta$-approximate minimizer of $\Phi$ corresponds to an $\eps$-approximate balancing. Bounding $L_i \leq 1$ and $\|x^*\|_2 \leq \sqrt{n} \diamG \log \condK$ by Corollary~\ref{cor:bal:R} therefore yields a bound of $O(n^{2} \eps^{-1} d \log \condK)$ iterations. Since iterations takes $O(m/n)$ time on average, this yields a final runtime bound of $O(mn\eps^{-1} \diamG \log \condK)$, which is not near-linear.
	}, essentially because they only rely on coordinate-smoothness of the function. 
	In order to achieve the near-linear time bounds in this paper, we heavily exploit the further global structure of the specific convex optimization problem at hand.
\end{remark}

\begin{remark}[$\ell_p$ Matrix Balancing]\label{rem:lp-bal}
	The (approximate) $\ell_p$ Matrix Balancing problem is: given input $K \in \C^{n \times n}$ and $p \in [1,\infty)$, compute a scaling $A = DKD^{-1}$ such that for each $i \in [n]$, the $i$-th row and columnn of $A$ have (approximately) equal $\ell_p$ norm. (Note this $\ell_p$ variant should not be confused with the norm discussion following~\eqref{eq-def:balance-approx}.) Note that the Matrix Balancing problem studied in this paper is a special case of this: it is $\ell_1$ balancing a nonnegative matrix. However, it is actually no less general, since $\ell_p$ balancing $K \in \C^{n \times n}$ is trivially reducible to $\ell_1$ balancing the nonnegative matrix with entries $|K_{ij}|^p$, see, e.g.,~\citep{RotSchSch94}. Thus, following the literature, we focus only on the version of Matrix Balancing described above.
\end{remark}

\begin{remark}[Max-Balancing]
	The Max-Balancing problem is $\ell_p$ Matrix Balancing for $p = \infty$, i.e.: given $K \in \Rpnn$, compute a scaling $A = DKD^{-1}$ so that for each $i$, the maximum entry in the $i$-th row and column of $A$ are equal. There is an extensive literature on this problem, including polynomial-time combinatorial algorithms~\citep{SchSch91,YouTarOrl91} as well as a natural analog of Osborne's algorithm~\citep{ParRei69} from the 1960s. Just as for Matrix Balancing, Osborne's algorithm has long been the choice in practice for Max-Balancing, yet its analysis has proven quite difficult: asymptotic convergence was not even known until 1998~\citep{Chen98}, and the first runtime bound was shown only a few years ago~\citep{SchSin17}. 
	However, despite the syntactic similarity of Max-Balancing and Matrix Balancing, the two problems are fundamentally very different: not only are the balancing goals different (which begets remarkably different properties, e.g., the Max-Balancing solution is not unique~\citep{Chen98}), but also the algorithms are quite different (even the analogous versions of Osborne's algorithm) and their analyses do not appear to carry over~\citep{OstRabYou17}.
\end{remark}

\begin{remark}[Matrix Scaling and Sinkhorn's algorithm]\label{rem:scaling}
	The Matrix Scaling problem is: given $K \in \Rpnn$ and vectors $\mu,\nu \in \Rpn$ satisfying $\sum_{i} \mu_i = \sum_i \nu_i$, find positive diagonal matrices $D_1,D_2$ such that $A := D_1KD_2$ satisfies $r(A) = \mu$ and $c(A) = \nu$. The many applications of Matrix Scaling have motivated an extensive literature on it; see, e.g., the survey~\citep{Idel16}. In analog to Osborne's algorithm for Matrix Balancing, there is a simple iterative procedure (Sinkhorn's algorithm) for Matrix Scaling~\citep{Sin67}.
	Sinkhorn's algorithm was recently shown to converge in near-linear time~\citep{AltWeeRig17} (see also~\citep{GurYia98,ChaKha18,DvuGasKro18}). The analysis there also uses a potential argument. Interestingly, the per-iteration potential improvement for Matrix Scaling is the \emph{Kullback-Leibler divergence} of the current imbalance, whereas for Matrix Balancing it is the \emph{Hellinger divergence}. Further connections related to algorithmic techniques in this paper are deferred to Appendix~\ref{app:sink}.
\end{remark}

\subsection{Roadmap}\label{ssec:intro:outline}
\S\ref{sec:prelim} recalls preliminary background. \S\ref{sec:pot} establishes the key lemmas in the potential argument. \S\ref{sec:greedy}, \S\ref{sec:rand}, \S\ref{sec:cyclic}, and \S\ref{sec:par} use these tools to prove fast convergence for Greedy, Random, Random-Reshuffle Cyclic, and parallelized Osborne variants, respectively.
For simplicity of exposition, these sections assume exact arithmetic; bit-complexity issues are addressed in \S\ref{sec:bits}. \S\ref{sec:conc} concludes with several open questions.

\section{Preliminaries}\label{sec:prelim}

\subsection{Notation}\label{ssec:prelim:notation}

For the convenience of the reader, we collect here the notation used commonly throughout the paper. We reserve $K \in \Rpnn$ for the matrix we seek to balance, $\eps > 0$ for the balancing accuracy, $m$ for the number of nonzero entries in $K$, $G_K$ for the graph associated to $K$, and $\diamG$ for the diameter of $G_K$. 
We assume throughout that the diagonal of $K$ is zero; this is without loss of generality because if $D$ solves the $\eps$-balancing problem for the matrix $K$ with zeroed-out diagonal, then $D$ solves the $\eps$-balancing problem for $K$.
The support, maximum entry, minimum nonzero entry, and condition number of $K$ are respectively denoted by $\supp(K) = \{(i,j) : K_{ij} > 0\}$, $\Kmax = \max_{ij} K_{ij}$, $\Kmin = \min_{(i,j) \in \supp(K)} K_{ij}$, and $\condK = (\sum_{ij}  K_{ij})/\Kmin$. 
The $\Otilde$ notation suppresses polylogarithmic factors in $n$ and $\eps$.
The all-ones and all-zeros vectors in $\Rn$ are respectively denoted by $\bone$ and $\zero$. Let $v \in \Rn$. The $\ell_1$ norm, $\ell_{\infty}$ norm, and variation semi-norm of $v$ are respectively $\|v\|_1 = \sum_{i=1}^n |v_i|$, $\|v\|_{\infty} = \max_{i \in [n]} |v_i|$, and $\var{v} = \max_i v_i - \min_j v_j$. 
We denote the entrywise exponentiation of $v$ by $e^v \in \R^n$, and the diagonalization of $v$ by $\diag(v) \in \Rnn$.
The set of discrete probability distributions on $n$ atoms is identified with the simplex $\Delta_n = \{p \in \Rpn : \sum_{i=1}^n p_i = 1 \}$. Let $\mu,\nu \in \Delta_n$. Their Hellinger distance is $\Hell(\mu,\nu) =  
\sqrt{ \frac{1}{2} \sum_{\ell=1}^n (\sqrt{\mu_\ell} - \sqrt{\nu_\ell})^2 }$, and their total variation distance is $\TV(\mu,\nu) = \|\mu - \nu\|_1/2$. 
We abbreviate ``with high probability'' by w.h.p., ``high probability'' by h.p., and ``almost surely'' by a.s.
We denote the minimum of $a,b \in \R$ by $a \wedge b$, and the maximum by $a \vee b$. 
Logarithms take base $e$ unless otherwise specified. 
All other specific notation is introduced in the main text.

\subsection{Matrix Balancing}\label{ssec:prelim:basic}

The formal definition of the (approximate) Matrix Balancing problem is in the ``log domain'' (i.e., output $x \in \Rn$ rather than $\diag(e^x)$). This is in part to avoid bit-complexity issues (see \S\ref{sec:bits}).

\begin{defin}[Matrix Balancing]\label{def:bal}
	The \emph{Matrix Balancing problem} $\BALK$ for input $K \in \Rpnn$ is to compute a vector $x \in \Rn$ such that $\diag(e^x) K \diag(e^{-x})$ is balanced. 
\end{defin}

\begin{defin}[Approximate Matrix Balancing]\label{def:abal}
	The \emph{approximate Matrix Balancing problem} $\ABALKeps$ for inputs $K \in \Rpnn$ and $\eps > 0$ is to compute a vector $x \in \Rn$ such that $\diag(e^x) K \diag(e^{-x})$ is $\eps$-balanced (see~\eqref{eq-def:balance}).
\end{defin}

$K \in \Rpnn$ is said to be \emph{balanceable} if $\BALK$ has a solution. It is known that non-balanceable matrices can be approximately balanced to arbitrary precision (i.e., $\ABAL$ has a solution for every $K \in \Rpnn$ and $\eps > 0$), and moreover that this is efficiently reducible to approximately balancing balanceable matrices, see, e.g.,~\citep{Chen00,CohMadTsiVla17}.
Thus, following the literature, we assume throughout that $K$ is balanceable. In the sequel, we make use of the following classical characterization of balanceable matrices in terms of their sparsity patterns. 

\begin{lemma}[Characterization of balanceability]\label{lem:balanceable}
	$K \in \Rpnn$ is balanceable if and only if it is irreducible---i.e., if and only if $G_K$ is strongly connected~\citep{Osborne60}. %
\end{lemma}

\subsection{Matrix Balancing as convex optimization}\label{ssec:prelim:convex}

Key to to our analysis---as well as much of the other Matrix Balancing literature (e.g.,~\citep{KalKhaSho97,OstRabYou17,CohMadTsiVla17,NemRot99})---is the classical connection between (approximately) balancing a matrix $K \in \Rpnn$ and (approximately) solving the convex optimization problem
\begin{align}
\min_{x \in \R^n} \Phi(x) := \log \sum_{ij} e^{x_i - x_j} K_{ij}.
\label{eq:opt}
\end{align}
In words, balancing $K$ is equivalent to scaling $DKD^{-1}$ so that the sum of its entries is minimized. This equivalence follows from KKT conditions and convexity of $\Phi(x)$, which ensures that local optimality implies global optimality. Intuition comes from computing the gradient:
\begin{align}
\nabla \Phi(x) = \frac{A\bone - A^T\bone}{\sum_{ij} A_{ij}}, \quad \text{where } A := \diag(e^x)K\diag(e^{-x}).
\label{eq:grad}
\end{align}
Indeed, solutions of $\BALK$ are points where this gradient vanishes, and thus are in correspondence with minimizers of $\Phi$.
This also holds approximately: solutions of $\ABALKeps$ are in correspondence with $\eps$-stationary points for $\Phi$ w.r.t. the $\ell_1$ norm, i.e., $x \in \Rn$ for which $\|\nabla \Phi(x)\|_1 \leq \eps$. The following lemma summarizes these classical connections; for a proof see, e.g.,~\citep{KalKhaSho97}.

\begin{lemma}[Matrix Balancing as convex optimization]\label{lem:bal:convex-opt}
	Let $K \in \Rpnn$ and $\eps > 0$. Then:
	\begin{enumerate}
		\item $\Phi$ is convex over $\Rn$. 
		\item $x \in \Rn$ is a solution to $\BALK$ if and only if $x$ minimizes $\Phi$.
		\item $x \in \Rn$ is a solution to $\ABALKeps$ if and only if $\|\nabla \Phi(x)\|_1 \leq \eps$.
		\item If $K$ is balanceable, then $\Phi$ has a unique minimizer modulo translations of $\bone$.
	\end{enumerate}
\end{lemma}

\subsection{Osborne's algorithm as coordinate descent}\label{ssec:prelim:cd}
Lemma~\ref{lem:bal:convex-opt} equates the \textit{problems} of (approximate) Matrix Balancing and (approximate) optimization of~\eqref{eq:opt}. This correspondence extends to \textit{algorithms}. In particular, in the sequel, we repeatedly leverage the following known connection, which appears in, e.g.,~\citep{OstRabYou17}.
\begin{obs}[Osborne's algorithm as Cordinate Descent]\label{obs:osb-cd}
	Osborne's algorithm for Matrix Balancing is equivalent to Exact Coordinate Descent for optimizing~\eqref{eq:opt}.
\end{obs}
To explain this connection, let us recall the basics of both algorithms. Exact Coordinate Descent is an iterative algorithm for minimizing a function $\Phi$ that maintains an iterate $x \in \Rn$, and in each iteration updates $x$ along a coordinate $k \in [n]$ by
\begin{align}
x
\gets
\argmin_{ z \in \{ x + \alpha e_k \, : \, \alpha \in \R \} } \Phi(z),
\label{eq:cd}
\end{align}
where $e_k$ denotes the $k$-th standard basis vector in $\R^n$. In words, this update~\eqref{eq:cd} improves the objective $\Phi(x)$ as much as possible by varying only the $k$-th coordinate of $x$. 
\par Osborne's algorithm, as introduced briefly in \S\ref{sec:intro}, is an iterative algorithm for Matrix Balancing that repeatedly balances row/column pairs. Algorithm~\ref{alg:osb} provides pseudocode for an implementation on the ``log domain'' that maintains the logarithms $x \in \Rn$ of the scalings rather than the scalings $\diag(e^x)$ themselves. The connection in Observation~\ref{obs:osb-cd} is thus, stated more precisely, that \emph{Osborne's algorithm is a specification of the Exact Coordinate Descent algorithm to minimizing the function $\Phi$ in~\eqref{eq:opt} with initialization of $\zero$.} This is because the Exact Coordinate Descent update to $\Phi$ on coordinate $k \in [n]$ updates $x_k$ so that $\frac{\partial \Phi}{\partial x_k}(x) = 0$, which by the derivative computation in~\eqref{eq:grad} amounts to updating $x_k$ so that the $k$-th row and column sums of the current balancing are equal---which is precisely the update rule for Osborne's algorithm on coordinate $k$.

\begin{algorithm}
	\caption{Osborne's algorithm for Matrix Balancing. The variant (e.g., Greedy, Random, or Random-Reshuffle Cyclic) depends on how the update coordinate is chosen in Line~\ref{line:osb-index}.}
	\hspace*{\algorithmicindent} \textbf{Input:} Matrix $K \in \Rpnn$ and accuracy $\eps > 0$ \\
	\hspace*{\algorithmicindent} \textbf{Output}: Vector $x \in \Rn$ that solves $\ABALKeps$
	\begin{algorithmic}[1]
		\State $x \gets \zero$ \Comment{Initialization}
		\While{$\diag(e^x)K\diag(e^{-x})$ is not $\eps$-balanced}
		\State Choose update coordinate $k \in [n]$ \label{line:osb-index}
		\State $x_k \gets x_k +
		\tfrac{\log(c_k( \diag(e^x)K\diag(e^{-x}) )) -\log(r_k( \diag(e^x)K\diag(e^{-x}) ))}{2}
		$ \Comment{Osborne update on coordinate $k$
		} \label{line:osb-update}
		\EndWhile
		\State \Return{$x$}
	\end{algorithmic}
	\label{alg:osb}
\end{algorithm}

We note that besides elucidating Observation~\ref{obs:osb-cd}, the log-domain implementation of Osborne's Algorithm in Algorithm~\ref{alg:osb} is also critical for numerical precision, both in theory and practice.

\begin{remark}[Log-domain implementation]\label{rem:logdomain}
	In practice, Osborne's algorithm should be implemented in the ``logarithmic domain'', i.e.,  store the iterates $x$ rather than the scalings $\diag(e^x)$, operate on $K$ through $\log K_{ij}$ (see Remark~\ref{rem:bits-log}), and compute Osborne updates using the following standard trick for numerically computing log-sum-exp: $\log ( \sum_{i=1}^n e^{z_i} ) = \max_j z_j + \log ( \sum_{i=1}^n e^{z_i - \max_j z_j} )$. In \S\ref{sec:bits}, we show that essentially just these modifications enable a provably logarithmic bit-complexity for several variants of Osborne's algorithm (Theorem~\ref{thm:osb-all:bits}).
\end{remark}

It remains to discuss the choice of update coordinate in Osborne's algorithm (Line~\ref{line:osb-index} of Algorithm~\ref{alg:osb}), or equivalently, in Coordinate Descent. We focus on the following natural options:
\begin{itemize}
	\item \textbf{Random-Reshuffle Cyclic Osborne.} Cycle through the coordinates, using an independent random permutation for the order each cycle.
	\item \textbf{Greedy Osborne.} Choose the coordinate $k$ for which the $k$-th row and column sums of the current scaling $A := \diag(e^x)K\diag(e^{-x})$ disagree most, as measured by
	\begin{align}
	\argmax_{k \in [n]} \abs{\sqrt{r_k(A)} - \sqrt{c_k(A)}}.
	\label{def:greedy}
	\end{align}
	(Ties are broken arbitrarily, e.g., lowest number.)
	\item \textbf{Random Osborne.} Sample $k$ uniformly from $[n]$, independently between iterations.
\end{itemize}

\begin{remark}[Efficient implementation of Greedy]\label{rem:greedy-amortized}
	In order to efficiently compute~\eqref{def:greedy}, Greedy Osborne maintains an auxiliary data structure: the row and column sums of the current balancing. This requires only $O(n)$ additional space, $O(m)$ additional computation in a pre-processing step, and $O(n)$ additional per-iteration computation for maintenance (increasing the per-iteration runtime by a small constant factor). 
\end{remark}

\subsection{Parallelizing Osborne's algorithm via graph coloring}\label{ssec:prelim:par}

For scalability, parallelization of Osborne's algorithm can be critical. It is well-known (see, e.g.,~\citep{BerTsi89}) that Osborne's algorithm can be parallelized when one can compute a (small) coloring of $G_K$, i.e., a partitioning $S_1, \dots, S_p$ of the vertices $[n]$ such that any two vertices in the same partitioning are non-adjacent.
This idea stems from the observation that simultaneous Osborne updates do not interfere with each other when performed on coordinates corresponding to \emph{non-adjacent} vertices in $G_K$. Indeed, this suggests a simple, natural parallelization of Osborne's algorithm given a coloring: update in parallel all coordinates of the same color. We call this algorithm \emph{Block Osborne} due to the following connection to Exact Block Coordinate Descent, i.e., the variant of Exact Coordinate Descent where an iteration exactly minimizes over a subset (a.k.a., block) of the variables.

\begin{remark}[Block Osborne as Block Coordinate Descent]\label{rem:par:block-cd}
	Extending Observation~\ref{obs:osb-cd}, \emph{Block} Osborne is equivalent to Exact \emph{Block} Coordinate Descent for minimizing $\Phi$. The connection to coloring is equivalently explained through this convex optimization lens: for each $S_{\ell}$, the (exponential\footnote{Note that by monotonocity of $\exp(\cdot)$, minimizing $\exp(\Phi(\cdot))$ is equivalent to minimizing $\Phi(\cdot)$.} of) $\Phi$ is separable in the variables in $S_{\ell}$. This is why their updates are independent.
\end{remark}

Just like the standard (non-parallelized) Osborne algorithm, the Block Osborne algorithm has several natural options for the choice of update block:
\begin{itemize}
	\item \textbf{Random-Reshuffle Cyclic Block Osborne.} Cycle through the blocks, using an independent random permutation for the order each cycle.
	\item \textbf{Greedy Block Osborne:} Choose the block $\ell$ maximizing
	\begin{align}
	\frac{1}{|S_{\ell}|}\sum_{k \in S_{\ell}} \left(\sqrt{r_k(A)} - \sqrt{c_k(A)}\right)^2
	\label{def:par-greedy}
	\end{align}
	where $A$ denotes the current balancing. (Ties are broken arbitrarily, e.g., lowest number.) 
	\item \textbf{Random Block Osborne.} Sample $\ell$ uniformly from $[p]$, independently between iterations.
\end{itemize}
Note that if $S_1,\dots,S_p$ are singletons---e.g., when $K \in \Rppnn$ is strictly positive---then these variants of Block Osborne degenerate into the corresponding variants of the standard Osborne algorithm.

Of course, Block Osborne first requires a coloring of $G_K$. A smaller coloring yields better parallelization (indeed we establish a linear runtime in the number of colors, see \S\ref{sec:par}). However, finding the (approximately) smallest coloring is NP-hard~\citep{Karp72,GarJoh76,Zuc06}. Nevertheless, in certain cases a relatively good coloring may be obvious or easily computable. For instance, in certain applications the sparsity pattern of $K$ could be structured, known a priori, and thus leveraged. An easily computable setting is matrices with uniformly sparse rows and columns, i.e., matrices whose corresponding graph $G_K$ has bounded max-degree; see Corollary~\ref{cor:par:unif}.

\section{Potential argument}\label{sec:pot}

Here we develop the ingredients for our potential-based analysis of Osborne's algorithm. They are purposely stated \emph{independently of the Osborne variant}, i.e., how the Osborne algorithm chooses update coordinates. This enables the argument to be applied directly to different variants in the sequel. We point the reader to \S\ref{ssec:intro:overview} for a high-level overview of the argument.

First, we recall the following standard bound on the initial potential. This appears in, e.g.,~\citep{OstRabYou17,CohMadTsiVla17}.
For completeness, we briefly recall the simple proof. Below, we denote the optimal value of the convex optimization problem~\eqref{eq:opt} by $\Phi^* := \min_{x \in \R^n} \Phi(x)$.

\begin{lemma}[Bound on initial potential]\label{lem:pot-init}
	$\Phi(\zero) - \Phi^* \leq \log \condK$.
\end{lemma}
\begin{proof}
	It suffices to show $\Phi^* \geq \log \Kmin$. Since $K$ is balanceable, $G_K$ is strongly connected (Lemma~\ref{lem:balanceable}), thus $G_K$ contains a cycle. By an averaging argument, this cycle contains an edge $(i,j)$ such that $x_i^* - x_j^* \geq 0$. Thus $\Phi^* \geq \log (e^{x_i^* - x_j^*} K_{ij}) \geq \log \Kmin$. 
\end{proof}

Next, we exactly compute the decrease in potential from an Osborne update on a fixed coordinate $k \in [n]$. This is a simple, direct calculation and is similar to~\citep[Lemma 2.1]{OstRabYou17}.

\begin{lemma}[Potential decrease from Osborne update]\label{lem:pot-dec}
	Consider any $x \in \R^n$ and update coordinate $k \in [n]$. Let $x'$ denote the output of an Osborne update on $x$ w.r.t. coordinate $k$, $A := \diag(e^x)K\diag(e^{-x})$ denote the scaling corresponding to $x$,  and $P := A/(\sum_{ij}A_{ij})$ its normalization. Then
	\begin{align}
	\Phi(x) - \Phi(x') = 
	- \log\left(1 - \left( \sqrt{r_k(P)} - \sqrt{c_k(P)} \right)^2 \right).
	\label{eq-lem:pot-dec}
	\end{align}
\end{lemma}
\begin{proof}
	Let $A' := \diag(e^{x'}) K \diag(e^{-x'})$ denote the scaling corresponding to the next iterate $x'$. Then $e^{\Phi(x)} - e^{\Phi(x')} = (r_k(A) + c_k(A)) - (r_k(A') + c_k(A')) = (r_k(A) + c_k(A)) - 2\sqrt{r_k(A)}\sqrt{c_k}(A) = (\sqrt{r_k(A)} - \sqrt{c_k(A)})^2 =  ( \sqrt{r_k(P)} - \sqrt{c_k(P)})^2 e^{\Phi(x)}$. Dividing by $e^{\Phi(x)}$ and re-arranging proves~\eqref{eq-lem:pot-dec}.
\end{proof}

In the sequel, we lower bound the per-iteration progress in~\eqref{eq-lem:pot-dec} by $(\sqrt{r_k(P)} - \sqrt{c_k(P)})^2$ using the elementary inequality $-\log(1 - z) \geq z$. Analyzing this further requires knowledge of how $k$ is chosen, i.e., the Osborne variant. However, for both Greedy Osborne and Random Osborne, this progress is at least the average
\begin{align}
\frac{1}{n}\sum_{k=1}^n (\sqrt{r_k(P)} - \sqrt{c_k(P)})^2 = \frac{2}{n}\Hell^2\big(r(P),c(P)\big).
\label{eq-pot:hell}
\end{align}
(For Random Osborne, this statement requires an expectation; see \S\ref{sec:rand}.) The rest of this section establishes the main ingredient in the potential argument: Proposition~\ref{prop:hell-lb} lower bounds this Hellinger imbalance, and thereby lower bounds the per-iteration progress. Note that Proposition~\ref{prop:hell-lb} is stated for ``nontrivial balancings'', i.e., $x \in \Rn$ satisfying $\Phi(x) \leq \Phi(\zero)$. This automatically holds for any iterate of the Osborne algorithm---regardless of the variant---since the first iterate is initialized to $\zero$, and since the potential is monotonically non-increasing by Lemma~\ref{lem:pot-dec}. 

\begin{prop}[Lower bound on Hellinger imbalance]\label{prop:hell-lb}
	Consider any $x \in \Rn$. Let $A := \diag(e^x) K \diag(e^{-x})$ denote the corresponding scaling, and let $P := A / \sum_{ij}A_{ij}$ denote its normalization. If $\Phi(x) \leq \Phi(\zero)$ and $A$ is not $\eps$-balanced, then
	\begin{align}
	\Hell^2\big(r(P),c(P)\big)
	\geq
	\frac{1}{8} \left( \frac{\Phi(x) - \Phi^*}{\diamG \log \condK} \vee \eps \right)^2.
	\label{eq-lem:pot-dec:rand}
	\end{align}
\end{prop}

To prove Proposition~\ref{prop:hell-lb}, we collect several helpful lemmas. The first is a standard inequality in statistics which lower bounds the Hellinger distance between two probability distributions by their $\ell_1$ distance (or equivalently, up to a factor of $2$, their total variation distance)~\citep{DezaDeza}. 
A short, simple proof via Cauchy-Schwarz is provided for completeness.

\begin{lemma}[Hellinger versus $\ell_1$ inequality]\label{lem:hell-ineq}
	If $\mu, \nu \in \Delta_n$, then
	\begin{align}
	\Hell(\mu,\nu) \geq \frac{1}{2\sqrt{2}} \|\mu - \nu\|_1.
	\end{align}
\end{lemma}
\begin{proof}
	By Cauchy-Schwarz, $
\|\mu - \nu \|_1^2
=
(\sum_k |\mu_k - \nu_k|)^2
=
(\sum_k |\sqrt{\mu_k} - \sqrt{\nu_k}| \cdot |\sqrt{\mu_k} + \sqrt{\nu_k}|)^2
\leq
(\sum_k (\sqrt{\mu_k} - \sqrt{\nu_k})^2) \cdot (\sum_k (\sqrt{\mu_k} + \sqrt{\nu_k})^2) = 2\Hell^2(\mu,\nu) \cdot ( \sum_k (\mu_k + \nu_k + 2\sqrt{\mu_k \nu_k}) )$. 
By the AM-GM inequality and the assumption $\mu,\nu \in \Delta_n$, the latter sum is at most $
\sum_k (\mu_k + \nu_k + 2\sqrt{\mu_k \nu_k}) \leq
2 \sum_k (\mu_k + \nu_k) = 4$.
\end{proof}

Next, we recall the following standard bound on the variation norm of nontrivial balancings. This bound is often stated only for optimal balancings (e.g.,~\citep[Lemma 4.24]{CohMadTsiVla17})---however, the proof extends essentially without modifications; details are provided briefly for completeness.

\begin{lemma}[Variation norm of nontrivial balancings]\label{lem:bal:R}
	If $x \in \Rn$ satisfies $\Phi(x) \leq \Phi(0)$, then $\var{x} \leq \diamG \log \condK$.
\end{lemma}
\begin{proof}
	Consider any $u,v \in [n]$. By definition of $\diamG$, there exists a path in $G_K$ from $u$ to $v$ of length at most $\diamG$. For each edge $(i,j)$ on the path, we have $e^{x_i - x_j} K_{ij} \leq \Phi(x) \leq \Phi(0)$, and thus $x_i - x_j \leq \log \condK$. Summing this inequality along the edges of the path and telescoping yields $x_u - x_v \leq d \log \condK$. Since this holds for any $u,v$, we conclude $\var{x} = \max_u x_u - \min_v x_v \leq d \log \condK$.
\end{proof}

From Lemma~\ref{lem:bal:R}, we deduce the following bound. 

\begin{cor}[$\ell_{\infty}$ distance of nontrivial balancings to minimizers]\label{cor:bal:R}
	If $x \in \Rn$ satisfies $\Phi(x) \leq \Phi(0)$, then there exists a minimizer $x^*$ of $\Phi$ such that $\|x - x^*\|_{\infty} \leq \diamG \log \condK$.
\end{cor}
\begin{proof}
	By definition, $\Phi$ is invariant under translations of $\bone$. Choose any minimizer $x^*$ and translate it by a multiple of $\bone$ so that $\max_i (x - x^*)_i = - \min_j (x - x^*)_j$. Then $\|x-x^*\|_{\infty} = (\max_i (x_i - x_i^*) - \min_j (x_j - x_j^*))/2 \leq ((\max_i x_i - \min_j x_j) + (\max_i x_i^* - \min_j x_j^*))/2 = (\var{x} + \var{x^*})/2$.
	  By Lemma~\ref{lem:bal:R}, this is at most $\diamG \log \condK$.
\end{proof}

We are now ready to prove Proposition~\ref{prop:hell-lb}.

\begin{proof}[Proof of Proposition~\ref{prop:hell-lb}]
	Since $P$ is normalized, its marginals $r(P)$ and $c(P)$ are both probability distributions in $\Delta_n$. Thus by Lemma~\ref{lem:hell-ineq},
	\begin{align}
	\Hell^2\big(r(P),c(P)\big)
	\geq
	\frac{1}{8} \|r(P) - c(P)\|_1^2.
	\label{eq-pf-hell-lb:main}
	\end{align}
	The claim now follows by lower bounding $\|r(P) - c(P)\|_1$ in two different ways. The first is $\|r(P) - c(P)\|_1 \geq \eps$, which holds since $A$ is not $\eps$-balanced by assumption. 
	The second is
	\begin{align}
 	\|r(P) - c(P)\|_1 \geq \frac{\Phi(x) - \Phi(x^*)}{\diamG \log \condK},
 	\label{eq-pf-hell-lb:2}
	\end{align}
	which we show presently. By convexity of $\Phi$ (Lemma~\ref{lem:bal:convex-opt}) and then H\"older's inequality, 
	\begin{align}
	\Phi(x) - \Phi(x^*) \leq \langle \nabla \Phi(x), x - x^* \rangle \leq \|\nabla \Phi(x)\|_1 \|x-x^*\|_{\infty}
	\label{eq-pf-hell-lb:convexity}
	\end{align}
	for any minimizer $x^*$ of $\Phi$. Now by Corollary~\ref{cor:bal:R}, there exists a minimizer $x^*$ such that $\|x - x^*\|_{\infty} \leq \diamG \log\condK$; and by~\eqref{eq:grad}, the gradient is $\nabla \Phi(x) = r(P) - c(P)$. Re-arranging~\eqref{eq-pf-hell-lb:convexity} therefore establishes~\eqref{eq-pf-hell-lb:2}.
\end{proof}

\section{Greedy Osborne converges quickly}\label{sec:greedy}

Here we show an improved runtime bound for Greedy Osborne that, for well-connected sparsity patterns, scales (near) linearly in both the total number of entries $n^2$ and the inverse accuracy $\eps^{-1}$. See \S\ref{ssec:intro:contributions} for further discussion of the result, and \S\ref{sssec:intro:overview-greedy} for a proof sketch.

\begin{theorem}[Convergence of Greedy Osborne]\label{thm:osb-greedy}
	Given a balanceable matrix $K \in \Rpnn$ and accuracy $\eps > 0$, Greedy Osborne 
	solves $\ABALKeps$
	in 
	$O(\tfrac{n^2}{\eps} (\tfrac{1}{\eps} \wedge \diamG)\log \condK)$ arithmetic operations.
\end{theorem}

The key lemma is that each iteration of Greedy Osborne improves the potential significantly. 

\begin{lemma}[Potential decrease of Greedy Osborne]\label{lem:pot-dec:greedy}
	Consider any $x \in \Rn$ for which the corresponding scaling $A := \diag(e^x) K \diag(e^{-x})$ is not $\eps$-balanced. If $x'$ is the next iterate obtained from a Greedy Osborne update, then
	\begin{align*}
	\Phi(x) - \Phi(x')
	\geq
	\frac{1}{4n} \left( \frac{\Phi(x) - \Phi^*}{\diamG \log \condK} \vee \eps \right)^2.
	\end{align*}
\end{lemma}
\begin{proof}
	Using in order Lemma~\ref{lem:pot-dec}, the inequality $-\log(1 - z) \geq z$ which holds for any $z \in \R$, the definition of Greedy Osborne, and then Proposition~\ref{prop:hell-lb},
	\begin{align}
	\Phi(x) - \Phi(x')
	&= 
	- \log (1 - \left( \sqrt{r_k(P)} - \sqrt{c_k(P)})^2 \right) 
	\label{eq-lpdg:1}
	\\ &\geq 
	\left( \sqrt{r_k(P)} - \sqrt{c_k(P)} \right)^2
	\label{eq-lpdg:2}
	\\ &\geq
	\frac{1}{n} \sum_{\ell=1}^n \left( \sqrt{r_{\ell}(P)} - \sqrt{c_{\ell}(P)} \right)^2
	\label{eq-lpdg:3}
	\\ &\geq
	\frac{1}{4n} \left( \frac{\Phi(x) - \Phi^*}{\diamG \log \condK} \vee \eps \right)^2. 
	\label{eq-lpdg:4}
	\end{align}
\end{proof}

\begin{proof}[Proof of Theorem~\ref{thm:osb-greedy}]
	Let $\xzero = \zero, \xone,\xtwo,\dots$ denote the iterates, and let $\tau$ be the first iteration for which $\diag(e^x)K\diag(e^{-x})$ is $\eps$-balanced. 
	Since the number of arithmetic operations per iteration is amortized to $O(n)$ by Remark~\ref{rem:greedy-amortized}, it suffices to show that the number of iterations $\tau$ is at most $O(n\eps^{-1}(\eps^{-1} \wedge \diamG) \log \condK)$. Now by Lemma~\ref{lem:pot-dec:greedy}, for each $t \in \{0,1,\dots,\tau-1\}$ we have
	\begin{align}
	\Phi(\xt) - \Phi(\xtp)
	\geq
	\frac{1}{4n} \left( \frac{\Phi(\xt) - \Phi^*}{\diamG \log \condK} \vee \eps \right)^2.
	\label{eq:greedy}
	\end{align}
	
	\paragraph*{Case 1: $\boldsymbol{\eps^{-1} \leq \diamG}$.} By the second bound in~\eqref{eq:greedy}, the potential decreases by at least $\eps^2/4n$ in each iteration. Since the potential is initially at most $\log \condK$ by Lemma~\ref{lem:pot-init} and is always nonnegative by definition, the total number of iterations is at most
	\begin{align}
	\tau \leq \frac{\log \condK}{\eps^2/4n} = \frac{4n \log \condK}{\eps^2}. 
	\label{eq-pf-greedy:c1}
	\end{align}
	
	\paragraph*{Case 2: $\boldsymbol{\eps^{-1} > \diamG}$.} For shorthand, denote $\alpha := \eps \diamG \log \condK$. Let $\tau_1$ be the first iteration for which the potential $\Phi(\xt) - \Phi^* \leq \alpha$, and let $\tau_2 := \tau - \tau_1$ denote the number of remaining iterations. By an identical argument as in case 1, 
	\begin{align}
	\tau_2
	\leq
	\frac{\alpha}{\eps^2/4n}
	=
	\frac{4n \diamG \log \condK}{\eps}.
	\label{eq-pf-greedy:c2-2}
	\end{align}
	To bound $\tau_1$, partition this phase further as follows. Let $\phi_0 := \log \condK$ and $\phi_i := \phi_{i-1}/2$ for $i = 1, 2, \dots$ until $\phi_N \leq \alpha$. Let $\tau_{1,i}$ be the number of iterations starting from when the potential is first no greater than $\phi_{i-1}$ and ending when it no greater than $\phi_{i}$. In the $i$-th subphase, the potential drops by at least $(\tfrac{\phi_i}{\diamG \log \condK})^2/4n$ per iteration by~\eqref{eq:greedy}. Thus
	\begin{align}
	\tau_{1,i} 
	\leq 
	\frac{\phi_{i-1} - \phi_i}{(\tfrac{\phi_i}{\diamG \log \condK})^2/4n}
	=
	\frac{4n \diamG^2 \log^2 \condK}{\phi_i}.
	\label{eq-pf-greedy:c2-1i}
	\end{align}
	Since $\sum_{i=1}^N \tfrac{1}{\phi_i} = \tfrac{1}{\phi_N} \sum_{j=0}^{N-1} 2^{-j} \leq \tfrac{2}{\phi_N} \leq \tfrac{4}{\alpha}$, thus
	\begin{align}
	\tau_1 
	=
	\sum_{i=1}^N \tau_{1,i}
	\leq 
	\frac{16n\diamG^2 \log \condK^2}{\alpha}
	=
	\frac{16n \diamG \log \condK}{\eps}.
	\label{eq-pf-greedy:c2-1}
	\end{align}
	By~\eqref{eq-pf-greedy:c2-2} and~\eqref{eq-pf-greedy:c2-1}, the total number of iterations is at most $\tau = \tau_1 + \tau_2 \leq 20n \diamG \eps^{-1} \log \condK$. 
\end{proof}

\section{Random Osborne converges quickly}\label{sec:rand}

Here we show that Random Osborne has runtime that is (i) near-linear in the input sparsity $m$; and (ii) also linear in the inverse accuracy $\eps^{-1}$ for well-connected sparsity patterns. See \S\ref{ssec:intro:contributions} for further discussion of the result, and \S\ref{sssec:intro:overview-rand} for a proof sketch.

\begin{theorem}[Convergence of Random Osborne]\label{thm:osb-rand}
	Given a balanceable matrix $K \in \Rpnn$ and accuracy $\eps > 0$, Random Osborne 
	solves $\ABALKeps$
	in $T$ arithmetic operations, where
	\begin{itemize}
		\item (Expectation guarantee.) $\E[T] = O(\tfrac{m}{\eps} (\tfrac{1}{\eps} \wedge \diamG) \log \condK)$.
		\item (H.p. guarantee.) There exists a universal constant $c > 0$ such that for all $\delta > 0$, 
		\[
		\Prob\left( T \leq c\left(\tfrac{m}{\eps} (\tfrac{1}{\eps} \wedge \diamG)  \log \condK \logdel \right)\right) \geq 1 - \delta.
		\]
	\end{itemize}
\end{theorem}

As described in the proof overview in \S\ref{sssec:intro:overview-greedy}, the core argument is nearly identical to the analysis of Greedy Osborne in \S\ref{sec:greedy}. Below, we detail the additional probabilistic nuances and describe how to overcome them. Remaining details for the proof of Theorem~\ref{thm:osb-rand} are deferred to Appendix~\ref{app:rand-pf}.

\subsection{Bounding the number of iterations}\label{sssec:conv:decrease}

Analogous to the proof of Greedy Osborne (c.f. Lemma~\ref{lem:pot-dec:greedy}), the key lemma is that each iteration significantly decreases the potential. The statement and proof are nearly identical. The only difference in the statement of the lemma is that for Random Osborne, this improvement is in \emph{expectation}.

\begin{lemma}[Potential decrease of Random Osborne]\label{lem:pot-dec:rand}
	Consider any $x \in \Rn$ for which the corresponding scaling $A := \diag(e^x) K \diag(e^{-x})$ is not $\eps$-balanced. If $x'$ is the next iterate obtained from a Random Osborne update, then
	\begin{align*}
	\E \left[  \Phi(x) - \Phi(x')  \right]
	\geq
	\frac{1}{4n} \left( \frac{\Phi(x) - \Phi^*}{\diamG \log \condK} \vee \eps \right)^2,
	\end{align*}
	where the expectation is over the algorithm's uniform random choice of update coordinate from $[n]$.
\end{lemma}
\begin{proof}
	The proof is identical to the proof for Greedy Osborne (c.f. Lemma~\ref{lem:pot-dec:greedy}), with only two minor differences. The first is that~\eqref{eq-lpdg:1} and~\eqref{eq-lpdg:2} are in expectation. The second is that~\eqref{eq-lpdg:3} holds with equality by definition of the Random Osborne algorithm.
\end{proof}

Lemma~\ref{lem:pot-init} shows that the potential is initially bounded, and Lemma~\ref{lem:pot-dec:rand} shows that each iteration significantly decreases the potential in expectation. In the analysis of Greedy Osborne, this potential drop is deterministic, and so we immediately concluded that the number of iterations is at most the initial potential divided by the per-iteration decrease (see~\eqref{eq-pf-greedy:c1} in \S\ref{sec:greedy}). Lemma~\ref{lem:prob} below shows that essentially the same bound holds in our stochastic setting. Indeed, the expectation bound is exactly this quantity (plus one), and the h.p. bound is the same up to a small constant.

\begin{lemma}[Per-iteration expected improvement implies few iterations]\label{lem:prob}
	Let $A > a$ and $h > 0$.
	Let $\{Y_t\}_{t \in \N_0}$ be a stochastic process adapted to a filtration $\{\cF_t\}_{t \in \N_0}$ such that $Y_0 \leq A$ a.s., 
	each difference $Y_{t-1} - Y_t$ is bounded within $[0,2(A-a)]$ a.s., and
	\begin{align}
	\E\left[ Y_t - Y_{t+1} \, | \, \cF_{t}, \, Y_{t} \geq a \right] \geq h
	\label{eq-lem:prob:iter}
	\end{align}
	for all $t \in \N_0$. Then the stopping time $\tau := \min \{t \in \N_0 \, : \, Y_t \leq a \}$ satisfies
	\begin{itemize}
		\item (Expectation bound.) $\E[\tau] \leq \tfrac{A - a}{h} + 1$.
		\item (H.p. bound.) 
		For all $\delta \in (0,1/e)$, it holds that $\Prob(\tau \leq \tfrac{6(A-a)}{h} \logdel ) \geq 1 - \delta$.
	\end{itemize}
\end{lemma}

The expectation bound in Lemma~\ref{lem:prob} is proved using Doob's Optional Stopping Theorem, and the h.p. bound using Chernoff bounds; details are deferred to Appendix~\ref{app:pfs-prob}.

\begin{remark}[Sub-exponential concentration]\label{rem:prob-subexp}
	Lemma~\ref{lem:prob} shows that the upper tail of $\tau$ decays at a sub-exponential rate. This concentration cannot be improved to a sub-Gaussian rate: indeed, consider $X_t$ i.i.d. Bernoulli with parameter $h \in (0,1)$, $Y_t = 1 - \sum_{i=1}^t X_i$, $A = 1$, and $a = 0$. Then $\Prob(\tau \leq N) = 1 - \Prob(X_1 = \dots = X_N = 0) = 1 - (1-h)^N$ which is $\approx 1 - \delta$ when $N \approx \tfrac{1}{h} \logdel$. 
\end{remark}

\subsection{Bounding the final runtime}\label{sssec:conv:runtime}

The key reason that Random Osborne is faster than Greedy Osborne (other than bit complexity) is that its per-iteration runtime is faster for sparse matrices: it is $O(m/n)$ by Observation~\ref{obs:osb-rand:iter} rather than $O(n)$.
In the deterministic setting, the final runtime is at most the product of the per-iteration runtime and the number of iterations (c.f. \S\ref{sec:greedy}). 
However, obtaining a final runtime bound from a per-iteration runtime and an iteration-complexity bound requires additional tools in the stochastic setting. 
A similar h.p. bound follows from a standard Chernoff bound. But proving an expectation bound is more nuanced.
The natural approach is Wald's equation, which states the the sum of a random number $\tau$ of i.i.d. random variables $Z_1,\dots,Z_\tau$ equals $\E\tau \E Z_1$, so long as $\tau$ is independent from $Z_1, \dots, Z_{\tau}$~\citep[Theorem 4.1.5]{Durrett10}. However, in our setting the per-iteration runtimes and the number of iterations are \emph{not} independent. Nevertheless, this dependence is weak enough for the identity to still hold. Formally, we require the following minor technical modifications of the per-iteration runtime bound in Observation~\ref{obs:osb-rand:iter} and Wald's equation. 

\begin{lemma}[Per-iteration runtime of Random Osborne, irrespective of history]\label{lem:osb-rand:iter}
	Let $\cF_{t-1}$ denote the sigma-algebra generated by the first $t-1$ iterates of Random Osborne. Conditional on $\cF_{t-1}$, the $t$-th iteration requires $O(m/n)$ arithmetic operations in expectation.
\end{lemma}

\begin{lemma}[Minor modification of Wald's equation]\label{lem:wald}
	Let $Z_1, Z_2, \dots$ be i.i.d. nonnegative integrable r.v.'s. Let $\tau$ be an integrable $\mathbb{N}$-valued r.v. satisfying $\E[Z_t | \tau \geq t ] = \E[Z_1]$ for each $t \in \mathbb{N}$. Then $\E[ \sum_{t=1}^{\tau} Z_t] = \E \tau \E Z_1$.
\end{lemma}

The proof of Lemma~\ref{lem:osb-rand:iter} is nearly identical to the proof of Observation~\ref{obs:osb-rand:iter}, and is thus omitted. 
The proof of Lemma~\ref{lem:wald} is a minor modification of the proof of the standard Wald's equation in~\citep{Durrett10}; details in Appendix~\ref{app:pfs-prob}.

\section{Random-Reshuffle Cyclic Osborne converges quickly}\label{sec:cyclic}

Here we show a runtime bound for Random-Reshuffle Cyclic Osborne. 
See \S\ref{ssec:intro:contributions} for further discussion, and \S\ref{sssec:intro:overview-cyc} for a proof sketch.

\begin{theorem}[Convergence of Random-Reshuffle Cyclic Osborne]\label{thm:osb-cyc}
	Given a balanceable matrix $K \in \Rpnn$ and accuracy $\eps > 0$, Random-Reshuffle Cyclic Osborne
	solves $\ABALKeps$
	in $T$ arithmetic operations, where
	\begin{itemize}
		\item (Expectation guarantee.) $\E[T] = O(\tfrac{mn}{\eps} (\tfrac{1}{\eps} \wedge \diamG) \log \condK)$.
		\item (H.p. guarantee.) There exists a universal constant $c > 0$ such that for all $\delta > 0$, 
		\[
		\Prob\left( T \leq c\left(\tfrac{mn}{\eps} (\tfrac{1}{\eps} \wedge \diamG)  \log \condK \logdel \right)\right) \geq 1 - \delta.
		\]
	\end{itemize}
\end{theorem}

A straightforward coupling argument with Random Osborne shows the following per-cycle potential decrease bound for Random-Reshuffle Cyclic Osborne.

\begin{lemma}[Potential decrease of Random-Reshuffle Cyclic Osborne]\label{lem:pot-dec:cyclic}
	Consider any $x \in \Rn$ for which the corresponding scaling $A := \diag(e^x) K \diag(e^{-x})$ is not $\eps$-balanced. Let $x'$ be the iterate obtained from $x$ after a cycle of Random-Reshuffle Cyclic Osborne. Then
	\begin{align*}
	\E \left[  \Phi(x) - \Phi(x')  \right]
	\geq
	\frac{1}{4n} \left( \frac{\Phi(x) - \Phi^*}{\diamG \log \condK} \vee \eps \right)^2,
	\end{align*}
	where the expectation is over the algorithm's random choice of update coordinates.
\end{lemma}
\begin{proof}
	By monotonicity of $\Phi$ w.r.t. Osborne updates (Lemma~\ref{lem:pot-dec}), the expected decrease in $\Phi$ from all $n$ updates in a cycle is at least that from the first update in the cycle. This first update index is uniformly distributed from $[n]$, thus is equivalent to an iteration of Random Osborne. We conclude by applying the per-iteration decrease bound for Random Osborne in Lemma~\ref{lem:pot-dec:rand}.
\end{proof}

The runtime bound for Random-Reshuffle Cyclic Osborne (Theorem~\ref{thm:osb-cyc}) given the expected per-cycle potential decrease (Lemma~\ref{lem:pot-dec:cyclic}) then follows by an identical argument as the runtime bound for Random Osborne (Theorem~\ref{thm:osb-rand}) given that algorithm's expected per-iteration potential decrease (Lemma~\ref{lem:pot-dec:rand}). The straightforward details are omitted for brevity.

\section{Parallelized variants of Osborne converge quickly}\label{sec:par}

Here we show fast runtime bounds for parallelized variants of Osborne's algorithm when given a coloring of $G_K$ (see \S\ref{ssec:prelim:par}). See \S\ref{ssec:intro:contributions} for a discussion of these results, and \S\ref{sssec:intro:overview-block} for a proof sketch.

\begin{theorem}[Convergence of Block Osborne variants]\label{thm:par:all}
	Consider balancing a balanceable matrix $K \in \Rpnn$ to accuracy $\eps > 0$ given a coloring of $G_K$ of size $p$. 
	\begin{itemize}
		\item Greedy Block Osborne solves $\ABALKeps$ in 	$O(\tfrac{p}{\eps} (\tfrac{1}{\eps} \wedge \diamG)\log \condK)$ rounds and $O(\tfrac{mp}{\eps} (\tfrac{1}{\eps} \wedge \diamG)\log \condK)$ total work.
		\item Random Block Osborne solves $\ABALKeps$ in $O(\tfrac{p}{\eps} (\tfrac{1}{\eps} \wedge \diamG) \log \condK)$ rounds and $O(\tfrac{m}{\eps} (\tfrac{1}{\eps} \wedge \diamG) \log \condK)$ total work, in expectation and w.h.p.
		\item Random-Reshuffle Cyclic Block Osborne solves $\ABALKeps$ in $O(\tfrac{p^2}{\eps} (\tfrac{1}{\eps} \wedge \diamG) \log \condK)$ rounds and $O(\tfrac{mp}{\eps} (\tfrac{1}{\eps} \wedge \diamG) \log \condK)$ total work, in expectation and w.h.p.
	\end{itemize} 
\end{theorem}

Note that the h.p. bounds in Theorem~\ref{thm:par:all} have exponentially decaying tails, just as for the non-parallelized variants (c.f., Theorems~\ref{thm:osb-rand} and~\ref{thm:osb-cyc}; see also Remark~\ref{rem:prob-subexp}).

The proof of Theorem~\ref{thm:par:all} is nearly identical to the analysis of the analogous non-parallelized variants in \S\ref{sec:greedy}, \S\ref{sec:rand}, and \S\ref{sec:cyclic} above. For brevity, we only describe the differences. First, we show the rounds bounds. For Greedy and Random Block Osborne, the only difference is that the per-iteration potential decrease is now $n/p$ times larger than in Lemmas~\ref{lem:pot-dec:greedy} and~\ref{lem:pot-dec:rand}, respectively. Below we show this modification for Greedy Block Osborne; an identical argument applies for Random Block Osborne after taking an expectation (the inequality~\eqref{eq:prop:pot-dec:greedy-block} then becomes an equality).

\begin{lemma}[Potential decrease of Greedy Block Osborne]\label{lem:pot-dec:greedy-block}
	Consider any $x \in \Rn$ for which the corresponding scaling $A := \diag(e^x) K \diag(e^{-x})$ is not $\eps$-balanced. If $x'$ is the next iterate obtained from a Greedy Block Osborne update, then
	\begin{align*}
	\Phi(x) - \Phi(x')
	\geq
	\frac{1}{4p} \left( \frac{\Phi(x) - \Phi^*}{\diamG \log \condK} \vee \eps \right)^2.
	\end{align*}
\end{lemma}
\begin{proof}
	Let $S_{\ell}$ be the chosen block. Using in order Lemma~\ref{lem:pot-dec}, the inequality $-\log(1 - z) \geq z$, the definition of Greedy Block Osborne, re-arranging, and then Proposition~\ref{prop:hell-lb},
	\begin{align}
	\Phi(x) - \Phi(x')
	&= 
	- \sum_{k \in S_{\ell}} \log (1 - \left( \sqrt{r_k(P)} - \sqrt{c_k(P)})^2 \right) \nonumber
	\\ &\geq 
	\sum_{k \in S_{\ell}} \left( \sqrt{r_k(P)} - \sqrt{c_k(P)} \right)^2 \nonumber
	\\ &\geq
	\frac{1}{p} \sum_{\ell=1}^p \sum_{k \in S_{\ell}} \left( \sqrt{r_{\ell}(P)} - \sqrt{c_{\ell}(P)} \right)^2
	\label{eq:prop:pot-dec:greedy-block}
	\\ &=
	\frac{1}{p} \sum_{k=1}^n \left( \sqrt{r_{k}(P)} - \sqrt{c_{k}(P)} \right)^2 \nonumber
	\\ &\geq
	\frac{1}{4p} \left( \frac{\Phi(x) - \Phi^*}{\diamG \log \condK} \vee \eps \right)^2. \nonumber
	\end{align}
\end{proof}

With this $n/p$ times larger per-iteration potential decrease, the number of rounds required by Greedy and Random Block Osborne is then $n/p$ times smaller than the number of Osborne updates required by their non-parallelized counterparts, establishing the desired rounds bounds in Theorem~\ref{thm:par:all}. The rounds bound for Random-Reshuffle Cyclic Block Osborne is then $p$ times that of Random Block Osborne by an identical coupling argument as for their non-parallelized counterparts (see \S\ref{sec:cyclic}). 
\par Next, we describe the total-work bounds in Theorem~\ref{thm:par:all}. For Random-Shuffle Cyclic Block Osborne, every $p$ rounds is a full cycle and therefore requires $\Theta(m)$ work. For Greedy and Random Block Osborne, each round takes work proportional to the number of nonzero entries in the updated block. For Random Block Osborne, this is $\Theta(m/p)$ on average by an identical argument to Observation~\ref{obs:osb-rand:iter}. For Greedy Block Osborne, this could be up to $O(m)$ in the worst case. (Although this is of course significantly improvable if the blocks have balanced sizes.)

Finally, we note that combining Theorem~\ref{thm:par:all} with the extensive literature on parallelized algorithms for coloring bounded-degree graphs yields a fast parallelized algorithm for balancing $\Delta$-uniformly sparse matrices, i.e., matrices $K$ for which $G_K$ has max degree\footnote{
	This is the degree in the undirected graph where $(i,j)$ is an edge if either $(i,j)$ or $(j,i)$ is an edge in $G_K$.
} $\Delta$.

\begin{cor}[Parallelized Osborne for uniformly sparse matrices]\label{cor:par:unif}
	There is a parallelized algorithm that, given any $\Delta$-uniformly sparse matrix $K \in \Rpnn$, computes an $\eps$-approximate balancing in $O(\tfrac{\Delta}{\eps} (\tfrac{1}{\eps} \wedge \diamG)\log \condK)$ rounds and $O(\tfrac{m}{\eps} (\tfrac{1}{\eps} \wedge \diamG)\log \condK)$ total work, both in expectation and w.h.p.
\end{cor}
\begin{proof}
	The algorithm of~\citep{BarElkKuh14} computes a $\Delta+1$  coloring in $O(\Delta) + \half \log^*n$ rounds, where $\log^*$ is the iterated logarithm. Run Random Block Osborne with this coloring, and apply Theorem~\ref{thm:par:all}.
\end{proof}

We remark that a coloring of size $\Delta+1$ can be alternatively computed by a simple greedy algorithm in $O(m)$ linear time. Although sequential, this simpler algorithm may be more practical.

\section{Numerical precision}\label{sec:bits}

So far we have assumed exact arithmetic for simplicity of exposition; here we address numerical precision issues. Note that Osborne iterates can have variation norm up to $O(n \log \condK)$; see~\citep[\S3]{KalKhaSho97} and Lemma~\ref{lem:bal:R}. For such iterates, operations on the current balancing $\diag(e^{x})K\diag(e^{-x})$---namely, computing row and column sums for an Osborne update---na\"ively require arithmetic operations on $O(n \log \condK)$-bit numbers. 
Here, we show that there is an implementation that uses numbers with only logarithmically few bits and still achieves the same runtime bounds.\footnote{
	Note that Theorem~\ref{thm:osb-all:bits} outputs only the balancing vector $x \in \R^n$, not the approximately balanced matrix $A = \diag(e^x)K\diag(e^{-x})$. If applications require $A$, this can be computed to polynomially small entrywise additive error using only logarithmically many bits; this is sufficient, e.g., for the application of approximating Min-Mean-Cycle~\citep[\S5.3]{AltPar20mmc}.}

\par Below, we assume for simplicity that each input entry $K_{ij}$ is represented using $O(\log \tfrac{\Kmax}{\Kmin} + \log \tfrac{n}{\eps})$ bits. (Or $O(\log \log \tfrac{\Kmax}{\Kmin} + \log \tfrac{n}{\eps})$ bits if input on the logarithmic scale $\log K_{ij}$, for $(i,j) \in \suppK$, see Remark~\ref{rem:bits-log}.) This assumption is made essentially without loss of generality since after a possible rescaling and truncation of entries to $\plusminus \eps\Kmin/n$---which does not change the problem of approximately balancing $K$ to $O(\eps)$ accuracy by Lemma~\ref{lem:bits:bal-app}---all inputs are represented using this many bits.

\begin{theorem}[Osborne variants with low bit-complexity]\label{thm:osb-all:bits}
	There is an implementation of Random Osborne (respectively, Random-Reshuffle Cyclic Osborne, Random Block Osborne, and Random-Reshuffle Cyclic Block Osborne) that uses arithmetic operations over $O(\log \tfrac{n}{\eps} +  \log \tfrac{\Kmax}{\Kmin})$-bit numbers and achieves the same runtime bounds as in Theorem~\ref{thm:osb-rand} (respectively, Theorem~\ref{thm:osb-cyc}, ~\ref{thm:par:all}, and again~\ref{thm:par:all}). 
\par Moreover, if the matrix $K$ is given as input through the logarithms of its entries $\{\log K_{ij}\}_{(i,j) \in \suppK}$, this bit-complexity is improvable to $O(\log \tfrac{n}{\eps} +  \log \log \tfrac{\Kmax}{\Kmin})$.
\end{theorem}

\par This result may be of independent interest since the aforementioned bit-complexity issues of Osborne's algorithm are well-known to cause numerical precision issues in practice and have been difficult to analyze theoretically. We note that~\citep[\S5]{OstRabYou17} shows similar bit complexity $O(\log (n \condK/\eps))$ for an Osborne variant they propose; however, that variant has runtime scaling in $n^2$ rather than $m$ (see footnote~\ref{ft:ost}). Moreover, our analysis is relatively simple and extends to the related Sinkhorn algorithm for Matrix Scaling (see Appendix~\ref{app:sink}).
\par Before proving Theorem~\ref{thm:osb-all:bits}, we make several remarks.

\begin{remark}[Log-domain input]\label{rem:bits-log}
	Theorem~\ref{thm:osb-all:bits} gives an improved bit-complexity if $K$ is input through the \textit{logarithms} of its entries. This is useful in an application such as Min-Mean-Cycle where the input is a weighted adjacency matrix $W$, and the matrix $K$ to balance is the entrywise exponential of (a constant times) $W$~\citep[\S5]{AltPar20mmc}. 
\end{remark}

\begin{remark}[Greedy Osborne requires large bit-complexity]\label{rem:greedy-bits}
	All known implementations of Greedy Osborne require bit-complexity at least $\tilde{\Omega}(n)$~\citep{OstRabYou17}. The obstacle is the computation~\eqref{def:greedy} of the next update coordinate, which requires computing the \emph{difference} of two log-sum-exp's. It can be shown that computing this difference to a constant multiplicative error suffices. However, this still requires at least computing the sign of the difference, which importantly, precludes dropping small summands in each log-sum-exp---a key trick used for computing an individual log-sum-exp to additive error with low bit-complexity (Lemma~\ref{lem:bits:lse}).
\end{remark}

We now turn to the proof of Theorem~\ref{thm:osb-all:bits}. For brevity, we establish this only for Random Osborne; the proofs for the other variants are nearly identical. Our implementation of Random Osborne makes three minor modifications to the exact-arithmetic implementation in Algorithm~\ref{alg:osb}. We emphasize that these modifications are in line with standard implementations of Osborne's algorithm in practice, see Remark~\ref{rem:logdomain}.

\begin{enumerate}
	\item In a pre-processing step, compute $\{\log K_{ij}\}_{(i,j) \in \suppK}$ to additive accuracy $\gamma = \Theta(\eps/n)$.
	\item Truncate each Osborne iterate $\xt$ entrywise to additive accuracy $\tau = \Theta(\eps^2/n)$.
	\item Compute Osborne updates to additive accuracy $\tau$ by using log-sum-exp computation tricks (Lemma~\ref{lem:bits:lse}) and using $K_{ij}$ only through the truncated values $\log K_{ij}$ computed in step $1$.
\end{enumerate}
Step 1 is performed only when $K$ is not already input on the logarithmic scale, and is responsible for the $O(\log (\Kmax/\Kmin))$ bit-complexity.
To argue about these modifications, we collect several helpful observations, the proofs of which are simple and deferred to Appendix~\ref{app:pf:bits} for brevity.

\begin{lemma}[Approximately balancing an approximate matrix suffices]\label{lem:bits:bal-app}
	Let $K,\tilde{K} \in \Rpnn$ such that $\supp(K) = \supp(\tilde{K})$ and 
 	the ratio $K_{ij}/\tilde{K}_{ij}$ of nonzero entries is bounded in $[1 - \gamma, 1 + \gamma]$ for some $\gamma \in (0,1/3)$. If $x$ is an $\eps$-balancing of $K$, then $x$ is an $(\eps + 6n\gamma)$-balancing of $\tilde{K}$.
\end{lemma}

\begin{lemma}[Stability of log-sum-exp]\label{lem:bits:lse-lip}
	The function $z \mapsto \log (\sum_{i=1}^n e^{z_i})$ is $1$-Lipschitz with respect to the $\ell_{\infty}$ norm on $\Rn$. 
\end{lemma}

\begin{lemma}[Stability of potential function]\label{lem:bits:phi-lip}
Let $K \in \Rpnn$. Then $\Phi(x) := \log (\sum_{ij}e^{x_i - x_j} K_{ij})$ is $2$-Lipschitz with respect to the $\ell_{\infty}$ norm on $\Rn$.
\end{lemma}

\begin{lemma}[Computing log-sum-exp with low bit-complexity]\label{lem:bits:lse}
	Let $z_1, \dots, z_n \in \R$ and $\tau > 0$ be given as input, each represented using $b $ bits. Then $\log(\sum_{i=1}^n e^{z_i})$ can computed to $\plusminus \tau$ in $O(n)$ operations on $O(b + \log(\tfrac{n}{\tau}))$-bit numbers.
\end{lemma}

\begin{proof}[Proof of Theorem~\ref{thm:osb-all:bits}]
\noindent\textbf{Error and runtime analysis.}
\begin{enumerate}
	\item Let $\tilde{K}$ be the matrix whose $ij$-th entry is the exponential of the truncated $\log K_{ij}$ for $(i,j) \in \suppK$, and $0$ otherwise. The effect of step (1) is to balance $\tilde{K}$ rather than $K$. But by Lemma~\ref{lem:bits:bal-app}, this suffices since an $O(\eps)$ balancing of $\tilde{K}$ is an $O(\eps + n \gamma) = O(\eps)$ balancing of $K$.
	\item[2,3.] The combined effect is that: given the previous Osborne iterate $\xtm$, the next iterate $\xt$ differs from the value it would have in the exact-arithmetic implementation by $O(\tau)$ in $\ell_{\infty}$ norm. By Lemma~\ref{lem:bits:phi-lip}, this changes $\Phi(\xt)$ by at most $O(\tau)$. By appropriately choosing the constant in the definition of $\tau = \Theta(\eps^2/n)$, this decreases each iteration's expected progress (Lemma~\ref{lem:pot-dec:rand}) by at most a factor of $1/2$. The proof of Theorem~\ref{thm:osb-rand} then proceeds otherwise unchanged, resulting in a final runtime at most $2$ times larger. 
\end{enumerate}
\textbf{Bit-complexity analysis.} 
\begin{enumerate}
	\item Consider $(i,j) \in \suppK$.
	Since $\log K_{ij} \in [\log \Kmin, \log \Kmax]$ and are stored to additive accuracy $\gamma = \Theta(\eps/n)$, the bit-complexity for storing $\log K_{ij}$ is 
	\[
	O\left( \log \frac{\log \Kmax - \log \Kmin}{\gamma} \right) = O\left(\log \frac{n}{\eps} + \log \log \frac{\Kmax}{\Kmin} \right).
	\]
	\item Since the coordinates of each Osborne iterate are truncated to additive accuracy $\tau = \Theta(\eps^2/n)$ and have modulus at most $\diamG \log \condK$ by Lemma~\ref{lem:bal:R}, they require bit-complexity
	\[
	O\left( \log \frac{(\diamG \log \condK) - (-\diamG \log \condK)}{\tau} \right)
	=
	O\left(\log \frac{n}{\eps} + \log \log \frac{\Kmax}{\Kmin} \right).
	\]
	\item By Lemma~\ref{lem:bits:lse}, the Osborne update requires bit-complexity $O(\log \frac{n}{\tau} ) = O(\log \frac{n}{\eps})$.
\end{enumerate}
\end{proof}

\section{Conclusion}\label{sec:conc}

We conclude with several open questions:

\begin{enumerate}
	\item Can one establish matching runtime lower bounds for the variants of Osborne's algorithm? The only existing lower bound is~\citep[Theorem 6.1]{OstRabYou17}, and there is a large gap between this and the current upper bounds. 

	\item Does any variant of Cyclic Osborne run in near-linear time? The best known runtime bound for Round-Robin Cyclic Osborne scales as roughly $mn^2$~\citep{OstRabYou17}, and the runtime bound we show for Random-Reshuffle Cyclic Osborne scales as roughly $mn$ (Theorem~\ref{thm:osb-cyc}).
	\item Is there a provable gap between the (worst-case) performance of Random Osborne, Random-Reshuffle Cyclic Osborne, and Round-Robin Cyclic Osborne? The existence of such gaps in the more general context of Coordinate Descent for convex optimization is an active area of research with recent breakthroughs~\citep{SunYe16,lee2019random,sun2020efficiency}.
	\item Empirically, Osborne's algorithm often significantly outperforms its worst-case bounds. 
	Is it possible to prove faster average-case runtimes for ``typical'' matrices arising in practice? (This is the analog to the third open question in~\citep[\S6]{SchSin17} for Max-Balancing.)
\end{enumerate}

\section*{Acknowledgements} JA thanks Enric Boix-Adser\`a and Jonathan Niles-Weed for helpful conversations.

\appendix

\section{Deferred proofs}\label{app:pfs}

\subsection{Probabilistic helper lemmas}\label{app:pfs-prob}

Several times we make use of the following standard (martingale) version of multiplicative Chernoff bounds, see, e.g.,~\citep[\S4]{MitUpf17}.

\begin{lemma}[Multiplicative Chernoff Bounds]\label{lem:chernoff}
	Let $X_1, \dots X_n$ be supported in $[0,1]$, be adapted to some filtration $\cF_0=\{\emptyset,\Omega \},\cF_1, \dots, \cF_n$, and satisfy $\E[X_i | \cF_{i-1}] = p$ for each $i \in [n]$. 
	Denote $X := \sum_{i=1}^n X_i$ and $\mu := \E X$. Then
	\begin{itemize}
		\item (Lower tail.) For any $\Delta \in (0,1)$, $
		\Prob\left( X \leq (1 - \Delta)\mu \right) \leq
		e^{-\Delta^2 \mu / 2}
		$.
		\item (Upper tail.) For any $\Delta \geq 1$,
		$
		\Prob\left( X \geq (1 + \Delta) \mu \right) \leq
		e^{-\Delta \mu / 3}
		$.
	\end{itemize}
\end{lemma}

\begin{proof}[Proof of Lemma~\ref{lem:prob}]
	\underline{Expectation bound.} Define $Z_t := Y_t + ht$. Then $Z_{t}^{\tau} := Z_{t \wedge \tau}$ is a stopped supermartingale with respect to $\cF_t$. Thus by Doob's Optional Stopping Theorem~\citep{Durrett10} (which may be invoked by a.s. boundedness), 
	\[
	A \geq \E Z_0 \geq \E Z_{\tau-1} = \E Y_{\tau-1} + h (\E \tau - 1) \geq a + h(\E\tau - 1)
	\]
	Re-arranging yields $\E[\tau] \leq \tfrac{A - a}{h} + 1$, as desired.
	\par \underline{High probability bound.} 
	For shorthand, denote $B := 2(A-a)$ and $N := \lceil 3B/h \logdel \rceil$.
	By definition of $\tau$, telescoping, and then the bound on $Y_0$,
	\begin{align}
	\Prob\left( \tau > N \right)
	=
	\Prob\left( Y_N > a \right)
	=
	\Prob\left( \sum_{t=1}^N (Y_{t-1} - Y_t) < Y_0 - a \right)
	\leq 
	\Prob\left( \sum_{t=1}^N (Y_{t-1} - Y_t) < A - a \right)
	\label{eq-pf:prob:hp-1}
	\end{align}
	To bound~\eqref{eq-pf:prob:hp-1}, define the process $X_t := (Y_{t-1} - Y_t) / B$. Each $X_t$ is a.s. bounded within $[0,1]$ by the bounded-difference assumption on $Y_t$. Thus by an application of the lower-tail Chernoff bound in Lemma~\ref{lem:chernoff} (combined with a simple stochastic domination argument since $\E[X_t | \cF_{t-1}] \geq h/B$ rather than exactly equal), and then the choice of $N$, we conclude that
	\begin{align}
	\Prob\left( \sum_{t=1}^N (Y_{t-1} - Y_t) < A - a \right)
	=
	\Prob\left( \sum_{t=1}^N X_t < \frac{A - a}{B} \right)
	\leq 
	\exp\left( -\left( 1 - \frac{A-a}{Nh} \right)^2 \frac{Nh}{2B} \right)
	\leq \delta.
	\label{eq-pf:prob:hp-2}
	\end{align}
\end{proof}

\begin{proof}[Proof of Lemma~\ref{lem:wald}]
	Observe that
	\begin{align*}
	\E\left[ \sum_{t=1}^{\tau} Z_t\right]
	&= 
	\sum_{T=1}^{\infty}
	\E\left[
	\sum_{t=1}^{\tau} Z_t \mathds{1}_{\tau = T}
	\right]
	= 
	\sum_{T=1}^{\infty}
	\sum_{t=1}^{T} 
	\E\left[
	Z_t  \mathds{1}_{\tau = T}
	\right]
	=
	\sum_{t=1}^{\infty}
	\sum_{T=t}^{\infty} 
	\E\left[
	Z_t \mathds{1}_{\tau = T}
	\right]
	=
	\sum_{t=1}^{\infty}
	\E\left[
	Z_t \mathds{1}_{\tau \geq t}
	\right],
	\end{align*}
	where the third equality above is because the assumption $Z_i \geq 0$ allows us to invoke Fubini's Theorem. Now since $
	\E\left[
	Z_t \mathds{1}_{\tau \geq t}
	\right]
	= 
	\E\left[
	Z_t | \tau \geq t
	\right] \Prob(\tau \geq t)
	= \E[Z_t] \Prob(\tau \geq t)$ by assumption, we conclude that $\E [\sum_{t=1}^{\tau} Z_t] = \E[Z_1] (\sum_{t=1}^{\infty} \Prob (\tau \geq t)) = \E[Z_1] \E[\tau]$. 
\end{proof}

\subsection{Proof of Theorem~\ref{thm:osb-rand}}\label{app:rand-pf}

Let $\xzero = \zero, \xone,\xtwo,\dots$ denote the iterates, and $\{\cF_t := \sigma(x_1, \dots, x_t)\}_t$ denote the corresponding filtration. Define the stopping time $\tau := \min\{t \in \N_0 : \diag(e^x)K\diag(e^{-x}) \text{ is } \eps\text{-balanced} \}$.
By Lemma~\ref{lem:pot-dec:rand},
\begin{align}
\E\left[ \Phi(\xt) - \Phi(\xtp) \, | \, \cF_t, t \leq \tau \right]
\geq
\frac{1}{4n} \left( \frac{\Phi(\xt) - \Phi^*}{\diamG \log \condK} \vee \eps \right)^2.
\label{eq-pf:pot-dec:rand}
\end{align}

\paragraph*{Case 1: $\boldsymbol{\eps^{-1} \leq \diamG}$.} 

Here, we establish the $O(m \eps^{-2} \log \condK)$ runtime bound both in expectation and w.h.p. To this end, let $T_t$ denote the runtime of iteration $t$, where (solely for analysis purposes) we consider also $t > \tau$ if the algorithm had continued after convergence. Define $Y_t$ to be $\Phi(\xt)$ if $t \leq \tau$, and otherwise $\Phi(\xt) - (t - \tau)\eps^2/4n$ if $t > \tau$. By~\eqref{eq-pf:pot-dec:rand}, we have
\begin{align}
\E\left[ Y_t - Y_{t+1} \, | \, \cF_t, Y_t \geq 0 \right] \geq \frac{\eps^2}{4n}.
\label{eq-pf:pot-dec:rand:case1}
\end{align}
For both expected and h.p. bounds below, we apply Lemma~\ref{lem:prob} to the process $Y_t$ with $A = \log \condK$ (by Lemma~\ref{lem:pot-init}), $a = 0$, and $h = \eps^2/4n$ (by~\eqref{eq-pf:pot-dec:rand:case1}).

\par \underline{Expectation bound.} 
The expectation bound in Lemma~\ref{lem:prob} implies
$\E[\tau] \leq 4n \eps^{-2} \log \condK + 1$. 
Since each iteration has expected runtime $\E[ T_t | \cF_{t-1}] = O(m/n)$ by Lemma~\ref{lem:osb-rand:iter}, Lemma~\ref{lem:wald} ensures that the total expected runtime is $\E T = \E [\sum_{t=1}^{\tau} T_t] = \E \tau \E T_1 = O(m \eps^{-2} \log \condK)$.
\par \underline{H.p. bound.} For shortand, denote $U := 24 n \eps^{-2} \log \condK \logtdel$. The h.p. bound in Lemma~\ref{lem:pot-init} implies that $\Prob(\tau > U) \leq \delta/2$.
By Lemma~\ref{lem:osb-rand:iter}, there is some constant $c > 0$ such that $\E [ T_t] = cm/n$. Since the $T_t$ are independent, a Chernoff bound (Lemma~\ref{lem:chernoff}) implies that $\Prob( \sum_{t=1}^U T_t \leq 2cUm/n) \leq \delta/2$.
Therefore, a union bound implies that with probability at least $1 - \delta$, the total runtime $T = \sum_{t=1}^{\tau} T_\tau$ is at most $2cUm/n = 48c m \eps^{-2} \log \condK \logtdel$.

\paragraph*{Case 2: $\boldsymbol{\eps^{-1} \geq \diamG}$.} Here, we establish the  $O(m \diamG \eps^{-1} \log \condK)$ runtime bound both in expectation and w.h.p. Define $\alpha, \tau_1$, $\tau_2$, $\tau_{1,i}$, and $\phi_i$ as in the analysis of Greedy Osborne (see \S\ref{sec:greedy}).

\par \underline{Expectation bound.} To bound $\E \tau_2$, define $Y_t$ and
apply Lemma~\ref{lem:prob} as in case $1$ above (except now with $A = \eps \diamG \log \condK$) to establish that
\begin{align}
\E\tau_2
\leq
\frac{\eps \diamG \log \condK}{\eps^2/4n} + 1	=
\frac{4n \diamG \log \condK}{\eps} + 1.
\label{eq-pf:rand:exp:2}
\end{align}
Next, we bound $\E \tau_1$. Consider subphase $\tau_{1,i}$ for $i \in [N]$. By an application of Lemma~\ref{lem:prob} on the process $\Phi(x^{(t - \tau_{1,i-1})})$ where $A = \phi_{i-1}$, $a = \phi_i$, and $h = \phi_i^2/(4n d^2 \log^2 \condK)$ from~\eqref{eq-pf:pot-dec:rand},
$
\E\tau_{1,i}
\leq
\frac{4n \diamG^2 \log^2 \condK}{\phi_i} + 1
$.
Thus $\E\tau_1
=
\sum_{i=1}^N
\E\tau_{1,i}
\leq
4n \diamG^2 \log^2 \condK (\sum_{i=1}^N \frac{1}{\phi_i}) + N$. Since $\sum_{i=1}^N \frac{1}{\phi_i} \leq \frac{4}{\eps \diamG \log \condK}$,
\begin{align}
\E\tau_1
\leq
\frac{16n \diamG \log \condK}{\eps} + \log_2 \left \lceil \frac{1}{\eps \diamG} \right \rceil .
\label{eq-pf:rand:exp:1}
\end{align}
Combining~\eqref{eq-pf:rand:exp:2} and~\eqref{eq-pf:rand:exp:1} establishes that
$	\E\tau= \E\tau_1 + \E\tau_2 \leq 21n \diamG \eps^{-1} \log \condK$.
By the $O(m/n)$ per-iteration expected runtime bound in Lemma~\ref{lem:osb-rand:iter} and the variant of Wald's equation in Lemma~\ref{lem:wald}, the total expected runtime is therefore at most $\E T \leq O(m/n) \cdot \E\tau = O(m \diamG \eps^{-1} \log \condK)$.

\par \underline{H.p. bound.} 
By Lemma~\ref{lem:prob}, $\Prob(\tau_2 > 24n \diamG \eps^{-1} \log \condK \log\tfrac{4}{\delta}) \leq \delta/4$. 	
To bound the first phase, define $p_i := \delta / 2^{N-i+3}$ for each $i \in [N]$. By Lemma~\ref{lem:prob}, $\Prob( \tau_{1,i} >
(24n\diamG^2\log^2 \condK \log 1/p_i)/\phi_i
) \leq p_i$. Note that $\sum_{i=1}^N \tfrac{\log 1/p_i}{\phi_i}
=
\tfrac{1}{\phi_N} \sum_{j=0}^{N-1} 2^{-j}( \log 8/\delta + j\log 2)
\leq 
\tfrac{1}{\phi_N} \sum_{j=0}^{\infty} 2^{-j}( \log 8/\delta + j \log 2)
=
\tfrac{2 \log 8/\delta + 2\log 2}{\phi_N}
\leq
\tfrac{6 \log 8/\delta}{ \eps \diamG \log \condK}$.
Thus by a union bound, with probability at most $\sum_{i=1}^N p_i \leq \delta/4$, the first phase has length at most $\tau_1 = \sum_{i=1}^N \tau_{1,i}
\leq 144 n \diamG \eps^{-1} \log \condK \log\tfrac{8}{\delta}$. 
We conclude by a further union bound that, with probability at least $1 - \delta/2$, the total number of iterations is at most $\tau = \tau_1 + \tau_2 \leq 168 n \diamG \eps^{-1} \log \condK \log\tfrac{8}{\delta}$. The proof is complete by an identical Chernoff bound argument as in case 1 above.

\subsection{Proofs for \S\ref{sec:bits}}\label{app:pf:bits}

\begin{proof}[Proof of Lemma~\ref{lem:bits:bal-app}]
	Let $A := \diag(e^x) K \diag(e^{-x})$ denote the corresponding scaling of $K$, and $P := A / \sum_{ij} A_{ij}$ denote its normalization. Similarly for $\tilde{A}$ and $\tilde{P}$. Note that each nonzero entry $\tilde{P}_{ij}$ approximates $P_{ij}$ to a multiplicative factor within $[(1-\gamma)/(1+\gamma), (1+\gamma)/(1-\gamma)] \subset [1-3\gamma, 1+3\gamma]$, where the last step used the assumption that $\gamma < 1/3$.
	Thus each row marginal $r_k(\tilde{P})$ approximates $r_k(P)$ to the same multiplicative factor, and similarly for the column marginals.
	Since $P$ and $\tilde{P}$ are normalized, this implies the additive approximations $|r_k(P) - r_k(\tilde{P})| \leq 3\gamma$, and similarly for the columns. Thus by the triangle inequality, $\|r(P) - c(P)\|_1 
	\leq 
	\|r(\tilde{P}) - c(\tilde{P})\|_1 
	+ 6n\gamma$.
\end{proof}

\begin{proof}[Proof of Lemma~\ref{lem:bits:lse-lip}]
	Let $x,y \in \R^n$. By the elementary inequality that $\min_i (a_i/b_i) \leq (\sum_{i=1}^n a_i) / (\sum_{i=1}^n b_i) \leq \max_i (a_i/b_i)$ for any $a,b \in \Rppn$, 
	\[
	\log \sum_{i=1}^n e^{x_i} - \log \sum_{i=1}^n e^{y_i}
	=
	\log \frac{\sum_{i=1}^n e^{x_i}}{\sum_{i=1}^n e^{y_i}}
	\leq
	\log \max_i e^{x_i - y_i}
	=
	\max_i x_i - y_i
	\leq
	\|x-y\|_{\infty},
	\]
	and similarly $\log \sum_{i=1}^n e^{x_i} - \log \sum_{i=1}^n e^{y_i} \geq \log \min_i e^{x_i - y_i} = \min_i x_i - y_i \geq -\|x-y\|_{\infty}$. We conclude that $|\log \sum_{i=1}^n e^{x_i} - \log \sum_{i=1}^n e^{y_i}| \leq \|x-y\|_{\infty}$.
\end{proof}

\begin{proof}[Proof of Lemma~\ref{lem:bits:phi-lip}]
	Let $x,y \in \Rn$. Clearly $|(x_i - x_j) - (y_i - y_j)| \leq 2\|x-y\|_{\infty}$ for any $i,j \in [n]$. Thus by Lemma~\ref{lem:bits:lse-lip}, $|\Phi(x) - \Phi(y)| = |\log (\sum_{(i,j) \in \suppK}e^{x_i - x_j + \log K_{ij}}) - \log (\sum_{(i,j) \in \suppK}e^{y_i - y_j + \log K_{ij}})| \leq  2\|x-y\|_{\infty}$.
\end{proof}

\begin{proof}[Proof of Lemma~\ref{lem:bits:lse}]
	Since $\log\sum_{i=1}^n e^{z_i} = \max_{j} z_j + \log \sum_{i=1}^n e^{z_i - (\max_j z_j)}$, we may assume without loss of generality after translation that each $z_i \leq 0$ and at least one $z_i = 0$. Since we need only approximate $\log \sum_{i=1}^n e^{z_i}$ to $\plusminus \tau$ accuracy, we can truncate each $z_i$ to additive accuracy $\plusminus O(\tau)$ by  Lemma~\ref{lem:bits:lse-lip}, and also drop all $z_i$ below $-\log\tfrac{n}{O(\tau)}$. To summarize, in order to compute $\log \sum_{i=1}^n e^{z_i}$ to $\plusminus \tau$, it suffices to compute $\log\sum_{i=1}^k e^{\tilde{z}_i}$ to $\plusminus O(\tau)$
	where $k \leq n$, each $\tilde{z}_i \in [-\log\tfrac{n}{O(\tau)}, 0]$, and each $\tilde{z}_i$ is represented by a number with at most $O(\log(\tfrac{\log (n/\tau)}{\tau} )) = O(\log\tfrac{1}{\tau} + \log \log n)$ bits. Now to compute $\log \sum_{i=1}^k e^{\tilde{z}_i}$ to $\plusminus O(\tau)$, we can tolerate computing each $e^{\tilde{z}_i}$ to multiplicative accuracy $(1\plusminus O(\tau))$. Thus since $e^{\tilde{z}_i} \geq O(\tau/n)$, we can tolerate computing each $e^{\tilde{z}_i}$ to additive accuracy $\plusminus O(\tau^2/n)$. Since $e^{\tilde{z}_i} \in
	[0,1]$, it therefore suffices to compute $e^{\tilde{z}_i}$ using $O(\log \tfrac{1}{\tau^2/n}) = O(\log \tfrac{n}{\tau})$ bits of precision. 
\end{proof}

\section{Connections to Matrix Scaling and Sinkhorn's algorithm}\label{app:sink}

Here, we continue the discussion in Remark~\ref{rem:scaling} by briefly mentioning two further connections between Osborne's algorithm for Matrix Balancing and Sinkhorn's algorithm for Matrix Scaling.
\\ \\ \noindent \textbf{Parallelizability.} In contrast to Osborne's algorithm for Matrix Balancing, Sinkhorn's algorithm for Matrix Scaling is so-called ``embarassingly parallelizable''. We briefly explain this in terms of the connection between parallelizability and graph coloring (see \S\ref{ssec:prelim:par}). For the Matrix Scaling problem on $K \in \Rp^{m \times n}$, the associated graph has vertex set $L \cup R$ where $|L| = m$ and $|R| = n$, and edge set $\{(i,j) : i \in [m], j \in [n], K_{ij} \neq 0 \}$. This graph is \emph{bipartite} and thus trivially \emph{$2$-colorable}, which is why Sinkhorn's algorithm can safely update all coordinates in $L$ or $R$ in parallel.
\\ \\ \noindent \textbf{Bit-complexity.} In Theorem~\ref{thm:osb-all:bits}, we showed that many variants of Osborne's algorithm can be implemented over numbers with logarithmically few bits, and still achieve the same runtime bounds. By a nearly identical argument, it can be shown that the analogous result applies to Sinkhorn's algorithm. This saves a similar factor of up to roughly $O(n)$ in the bit-complexity for poorly connected inputs. Moreover, this modification is also helpful for well-connected inputs, in particular for the application of Optimal Transport, where the matrix $K$ to scale is dense yet has exponentially large entries which require bit-complexity $O(L(\log n)/\eps)$ in the notation of~\citep[Remark 1]{AltWeeRig17}. This modification reduces the bit-complexity to only logarithmic size $O(\log(Ln/\eps))$.

\footnotesize
\addcontentsline{toc}{section}{References}
\bibliographystyle{abbrv}
\bibliography{bal}{}

\begin{thebibliography}{10}

\bibitem{ZhuLiOliWig17}
Z.~Allen-Zhu, Y.~Li, R.~Oliveira, and A.~Wigderson.
\newblock Much faster algorithms for matrix scaling.
\newblock In {\em Symposium on the Foundations of Computer Science (FOCS)}.
  IEEE, 2017.

\bibitem{AllQuRicYua16}
Z.~Allen-Zhu, Z.~Qu, P.~Richt{\'a}rik, and Y.~Yuan.
\newblock Even faster accelerated coordinate descent using non-uniform
  sampling.
\newblock In {\em International Conference on Machine Learning (ICML)}, pages
  1110--1119, 2016.

\bibitem{AltWeeRig17}
J.~Altschuler, J.~Weed, and P.~Rigollet.
\newblock Near-linear time approximation algorithms for optimal transport via
  {S}inkhorn iteration.
\newblock In {\em Conference on Neural Information Processing Systems
  (NeurIPS)}, 2017.

\bibitem{AltPar20mmc}
J.~M. Altschuler and P.~A. Parrilo.
\newblock Approximating {M}in-{M}ean-{C}ycle for low-diameter graphs in
  near-optimal time and memory.
\newblock In {\em arXiv pre-print}, 2020.

\bibitem{Lapack}
E.~Anderson, Z.~Bai, C.~Bischof, S.~Blackford, J.~Demmel, J.~Dongarra,
  J.~Du~Croz, A.~Greenbaum, S.~Hammarling, A.~McKenney, and D.~Sorensen.
\newblock {\em {LAPACK} Users' Guide}.
\newblock Society for Industrial and Applied Mathematics, Philadelphia, PA,
  third edition, 1999.

\bibitem{BarElkKuh14}
L.~Barenboim, M.~Elkin, and F.~Kuhn.
\newblock Distributed (${\Delta}$+1)-coloring in linear (in ${\Delta}$) time.
\newblock {\em SIAM Journal on Computing}, 43(1):72--95, 2014.

\bibitem{BerTsi89}
D.~P. Bertsekas and J.~N. Tsitsiklis.
\newblock {\em Parallel and distributed computation: numerical methods},
  volume~23.
\newblock Prentice Hall Englewood Cliffs, NJ, 1989.

\bibitem{Bol01}
B.~Bollob{\'a}s.
\newblock {\em Random graphs}.
\newblock Number~73. Cambridge University Press, 2001.

\bibitem{ChaKha18}
D.~Chakrabarty and S.~Khanna.
\newblock Better and simpler error analysis of the {S}inkhorn-{K}nopp algorithm
  for matrix scaling.
\newblock In {\em Symposium on Simplicity in Algorithms (SOSA)}, 2018.

\bibitem{Chen98}
T.-Y. Chen.
\newblock Balancing sparse matrices for computing eigenvalues.
\newblock Master's thesis, UC Berkeley, 5 1998.

\bibitem{Chen00}
T.-Y. Chen and J.~W. Demmel.
\newblock Balancing sparse matrices for computing eigenvalues.
\newblock {\em Linear Algebra and its Applications}, 309(1-3):261--287, 2000.

\bibitem{CohMadTsiVla17}
M.~B. Cohen, A.~Madry, D.~Tsipras, and A.~Vladu.
\newblock Matrix scaling and balancing via box constrained {N}ewton's method
  and interior point methods.
\newblock In {\em Symposium on the Foundations of Computer Science (FOCS)},
  pages 902--913. IEEE, 2017.

\bibitem{DezaDeza}
M.~M. Deza and E.~Deza.
\newblock Encyclopedia of distances.
\newblock pages 1--583. Springer, 2009.

\bibitem{Durrett10}
R.~Durrett.
\newblock {\em Probability: theory and examples}.
\newblock Cambridge University Press, 2010.

\bibitem{DvuGasKro18}
P.~Dvurechensky, A.~Gasnikov, and A.~Kroshnin.
\newblock Computational optimal transport: complexity by accelerated gradient
  descent is better than by {S}inkhorn's algorithm.
\newblock In {\em International Conference on Machine Learning (ICML)}, 2018.

\bibitem{EavHofRotSch85}
B.~C. Eaves, A.~J. Hoffman, U.~G. Rothblum, and H.~Schneider.
\newblock Line-sum-symmetric scalings of square nonnegative matrices.
\newblock pages 124--141, 1985.

\bibitem{GarJoh76}
M.~R. Garey and D.~S. Johnson.
\newblock The complexity of near-optimal graph coloring.
\newblock {\em Journal of the ACM}, 23(1):43--49, 1976.

\bibitem{Rexpm}
V.~Goulet, C.~Dutang, M.~Maechler, D.~Firth, M.~Shapira, M.~Stadelmann, et~al.
\newblock expm: Matrix exponential.
\newblock {\em R package version 0.99-0}, 2013.

\bibitem{GurYia98}
L.~Gurvits and P.~N. Yianilos.
\newblock The deflation-inflation method for certain semidefinite programming
  and maximum determinant completion problems.
\newblock Technical report, NECI, 1998.

\bibitem{Higham05}
N.~J. Higham.
\newblock The scaling and squaring method for the matrix exponential revisited.
\newblock {\em SIAM Journal on Matrix Analysis and Applications},
  26(4):1179--1193, 2005.

\bibitem{Idel16}
M.~Idel.
\newblock A review of matrix scaling and {S}inkhorn's normal form for matrices
  and positive maps.
\newblock {\em arXiv preprint arXiv:1609.06349}, 2016.

\bibitem{KalKhaSho97}
B.~Kalantari, L.~Khachiyan, and A.~Shokoufandeh.
\newblock On the complexity of matrix balancing.
\newblock {\em SIAM Journal on Matrix Analysis and Applications},
  18(2):450--463, 1997.

\bibitem{Karp72}
R.~M. Karp.
\newblock Reducibility among combinatorial problems.
\newblock In {\em Complexity of Computer Computations}, pages 85--103.
  Springer, 1972.

\bibitem{lee2019random}
C.-P. Lee and S.~J. Wright.
\newblock Random permutations fix a worst case for cyclic coordinate descent.
\newblock {\em IMA Journal of Numerical Analysis}, 39(3):1246--1275, 2019.

\bibitem{MaiBat19}
V.~S. Mai and A.~Battou.
\newblock Asynchronous distributed matrix balancing and application to
  suppressing epidemic.
\newblock In {\em 2019 American Control Conference (ACC)}, pages 2177--2182.
  IEEE, 2019.

\bibitem{MATLABbal}
MathWorks.
\newblock balance: diagonal scaling to improve eigenvalue accuracy.
\newblock {\em https://www.mathworks.com/help/matlab/ref/balance.html}.

\bibitem{MATLABeig}
MathWorks.
\newblock eig: eigenvalues and eigenvectors.
\newblock {\em https://www.mathworks.com/help/matlab/ref/eig.html}.

\bibitem{MitUpf17}
M.~Mitzenmacher and E.~Upfal.
\newblock {\em Probability and {C}omputing: randomization and probabilistic
  techniques in algorithms and data analysis}.
\newblock Cambridge University Press, 2017.

\bibitem{NemRot99}
A.~Nemirovski and U.~Rothblum.
\newblock On complexity of matrix scaling.
\newblock {\em Linear Algebra and its Applications}, 302:435--460, 1999.

\bibitem{NesSti17}
Y.~Nesterov and S.~U. Stich.
\newblock Efficiency of the accelerated coordinate descent method on structured
  optimization problems.
\newblock {\em SIAM Journal on Optimization}, 27(1):110--123, 2017.

\bibitem{Osborne60}
E.~Osborne.
\newblock On pre-conditioning of matrices.
\newblock {\em Journal of the ACM}, 7(4):338--345, 1960.

\bibitem{OstRabYou17}
R.~Ostrovsky, Y.~Rabani, and A.~Yousefi.
\newblock Matrix balancing in $l_p$ norms: bounding the convergence rate of
  {O}sborne's iteration.
\newblock In {\em Symposium on Discrete Algorithms (SODA)}, pages 154--169.
  SIAM, 2017.

\bibitem{OstRabYou18}
R.~Ostrovsky, Y.~Rabani, and A.~Yousefi.
\newblock Strictly balancing matrices in polynomial time using {O}sborne's
  iteration.
\newblock In {\em International Colloquium on Automata, Languages and
  Programming (ICALP)}, 2018.

\bibitem{ParRei69}
B.~N. Parlett and C.~Reinsch.
\newblock Balancing a matrix for calculation of eigenvalues and eigenvectors.
\newblock {\em Numerische Mathematik}, 13(4):293--304, 1969.

\bibitem{PreTeuVetFla07}
W.~H. Press, S.~A. Teukolsky, W.~T. Vetterling, and B.~P. Flannery.
\newblock {\em Numerical recipes 3rd edition: The art of scientific computing}.
\newblock Cambridge University Press, 2007.

\bibitem{Rbal}
RDocumentation.
\newblock Balance a square matrix via {LAPACK}'s dgebal.
\newblock {\em
  https://www.rdocumentation.org/packages/expm/versions/0.99-1.1/topics/balance}.

\bibitem{RotSchSch94}
U.~G. Rothblum, H.~Schneider, and M.~H. Schneider.
\newblock Scaling matrices to prescribed row and column maxima.
\newblock {\em SIAM Journal on Matrix Analysis and Applications}, 15(1):1--14,
  1994.

\bibitem{SchSch91}
H.~Schneider and M.~H. Schneider.
\newblock Max-balancing weighted directed graphs and matrix scaling.
\newblock {\em Mathematics of Operations Research}, 16(1):208--222, 1991.

\bibitem{SchZen90}
M.~H. Schneider and S.~A. Zenios.
\newblock A comparative study of algorithms for matrix balancing.
\newblock {\em Operations Research}, 38(3):439--455, 1990.

\bibitem{SchSin17}
L.~J. Schulman and A.~Sinclair.
\newblock Analysis of a classical matrix preconditioning algorithm.
\newblock {\em Journal of the ACM}, 64(2):9, 2017.

\bibitem{Sin67}
R.~Sinkhorn.
\newblock Diagonal equivalence to matrices with prescribed row and column sums.
\newblock {\em The American Mathematical Monthly}, 74(4):402--405, 1967.

\bibitem{Eispack}
B.~T. Smith, J.~M. Boyle, B.~Garbow, Y.~Ikebe, V.~Klema, and C.~Moler.
\newblock {\em Matrix eigensystem routines - {EISPACK} guide}, volume~6.
\newblock Springer, 2013.

\bibitem{sun2020efficiency}
R.~Sun, Z.-Q. Luo, and Y.~Ye.
\newblock On the efficiency of random permutation for {ADMM} and coordinate
  descent.
\newblock {\em Mathematics of Operations Research}, 45(1):233--271, 2020.

\bibitem{SunYe16}
R.~Sun and Y.~Ye.
\newblock Worst-case complexity of cyclic coordinate descent: ${O}(n^2)$ gap
  with randomized version.
\newblock {\em Mathematical Programming}, pages 1--34, 2019.

\bibitem{TarjanStronglyConnected}
R.~Tarjan.
\newblock Depth-first search and linear graph algorithms.
\newblock {\em SIAM Journal on Computing}, 1(2):146--160, 1972.

\bibitem{Tom03}
J.~A. Tomlin.
\newblock A new paradigm for ranking pages on the world wide web.
\newblock In {\em Proceedings of the 12th international conference on World
  Wide Web}, pages 350--355, 2003.

\bibitem{Ward77}
R.~C. Ward.
\newblock Numerical computation of the matrix exponential with accuracy
  estimate.
\newblock {\em SIAM Journal on Numerical Analysis}, 14(4):600--610, 1977.

\bibitem{YouTarOrl91}
N.~E. Young, R.~E. Tarjan, and J.~B. Orlin.
\newblock Faster parametric shortest path and minimum-balance algorithms.
\newblock {\em Networks}, 21(2):205--221, 1991.

\bibitem{Zuc06}
D.~Zuckerman.
\newblock Linear degree extractors and the inapproximability of max clique and
  chromatic number.
\newblock In {\em Symposium on the Theory of Computing (STOC)}, pages 681--690.
  ACM, 2006.

\end{thebibliography}

\end{document}